\documentclass[11pt]{article}

\usepackage[utf8]{inputenc}
\usepackage[T1]{fontenc}
\usepackage[hidelinks]{hyperref}
\usepackage{amsmath}
\usepackage{amssymb}
\usepackage{amsthm}
\usepackage{graphicx}
\usepackage{caption}
\newtheorem{proposition}{Proposition}

\newtheorem*{remark}{Remark}

\usepackage{physics}
\usepackage{comment}
\usepackage{authblk}
\usepackage{xcolor}
\usepackage{subcaption}
\newcommand{\R}{\mathbb{R}}
\newcommand{\Z}{\mathbb{Z}}
\newcommand{\kin}{\mathcal{E}_{kin}}

\title{Conservative stabilized Runge-Kutta methods for the Vlasov-Fokker-Planck equation}
\author[a]{Ibrahim Almuslimani}
\author[b]{Nicolas Crouseilles}
\affil[a]{Univ Rennes, Inria (MINGuS)  \& IRMAR UMR $6625$, France}
\affil[b]{Univ Rennes, Inria (MINGuS)  \& IRMAR UMR $6625$ \&  ENS Rennes, France}

\date{}
\begin{document}

\maketitle


\begin{abstract}
In this work, we aim at constructing numerical schemes, that are as efficient as possible in terms of cost and conservation of invariants, for the Vlasov--Fokker--Planck system coupled with Poisson or Amp\`ere equation. Splitting methods are used where the linear terms in space are treated by spectral or semi-Lagrangian methods and the nonlinear diffusion in velocity in the collision operator is treated using a stabilized Runge--Kutta--Chebyshev (RKC) integrator, a powerful alternative of implicit schemes. The new schemes are shown to exactly preserve mass and momentum. The conservation of total energy is obtained using a suitable approximation of the electric field. An H-theorem is proved in the semi-discrete case, while the entropy decay is illustrated numerically for the fully discretized problem. Numerical experiments that include investigation of Landau damping phenomenon and bump-on-tail instability are performed to illustrate the efficiency of the new schemes.

\bigskip

\noindent
{\it Keywords:\,}
explicit stabilized integrators, RKC methods, kinetic equations, Vlasov--Poisson system, Vlasov--Fokker--Planck equation, Landau damping.
\smallskip

\noindent
{\it AMS subject classification (2010):\,}
35Q83, 35Q84, 65M12.
\end{abstract}

	\section{Introduction}
We consider the following model satisfied by $f(t, x, v)$ with the time $t\geq 0$, the space variable $x\in\mathbb{R}^{d_x}$ 
and the velocity variable $v\in\mathbb{R}^{d_v}$ 
\begin{equation}
\label{vp_fp}
\partial_t f + v\cdot \nabla_x f + E\cdot \nabla_v f =\nu \nabla_v \cdot ((v-u_f)f + T_f\nabla_v f) =:\nu Q(f), 
\end{equation}
where $\nu>0$ is the collision rate, the electric field $E(t, x)=-\nabla_x\phi(t, x)$ is obtained from the Poisson equation 
$\Delta_x \phi(t, x) = \int_{\mathbb{R}^{d_v}} f(t, x, v) dv -1$, the mean velocity $u_f(t, x)$ and temperature $T_f(t, x)$ 
are defined by 
\begin{eqnarray}
\label{eq:Tu}
n_f u_f(t, x) &=& \int_{\mathbb{R}^{d_v}} v f(t, x, v) dv, \nonumber\\
n_f T_f(t, x) &=&\frac{1}{d_v}\int_{\mathbb{R}^{d_v}} |v-u_f(t, x)|^2 f(t, x, v) dv\nonumber\\
\mbox{with the density } && n_f(t, x)=\int_{\mathbb{R}^{d_v}} f(t, x, v) dv. 
\end{eqnarray} 

The goal of this work is to construct efficient numerical grid based schemes (ie using a grid of the phase space) 
for the numerical approximation of \eqref{vp_fp}.  
Even if these methods are costly, their high accuracy enables a very good description of the phase space even in low density 
regions. Moreover, stability conditions relating space or velocity mesh to the time step usually lead to high CPU time, 
in particular when dissipative collision operators are considered. 
To overcome this drawback, we aim at using stabilized Runge-Kutta-Chebyshev (RKC) methods for the numerical simulations 
of collisional Vlasov equations  \eqref{vp_fp}.   

RKC methods are explicit methods with an extended stability domain along the real negative axis which are very well suited 
for the time integration of  problems whose eigenvalues lie in a narrow strip around the negative real axis, 
which is typically the case for parabolic problems. Indeed, when one is interested in high dimensional or/and nonlinear parabolic problems, 
implicit method require to solve large or/and nonlinear systems which may become very costly from a computational point of view. 
Inversely, explicit methods do not have to invert complex systems, but at the cost of a severe condition on  the time step. 
Additionally, when the system comes from the discretization of a PDE which enjoys suitable properties, 
explicit methods may be preferred since the properties can be easily extended to the fully discrete case, and 
implicit methods may break the properties of the underlying numerical schemes.  
Hence, RKC methods turn out to be a good compromise since they are explicit and their domain of stability has a length proportional to $s^2$ with $s$ is the number of stages, which allow for the use of large time steps (see for instance \cite{Abdrev,SSV98,HaW96}). 

Compared with high-order RK schemes, stabilized RKC schemes require few function calls, 
and have an extended stability region along the negative real axis, which covers the eigenvalues of the dissipative collision term more efficiently and allows larger time steps. Even if the error is of second order, the global error is relatively small 
when the full system is considered. In addition, to optimize time stepping, the second order RKC scheme is coupled with adaptive 
time stepping strategy: the time step is automatically chosen according to a local error estimate, and the number of stages is dynamically adjusted.

In the context of Vlasov equation, solving large systems is not acceptable due to the high dimensionality of the equations. 
Standard explicit methods, due to the fact that they require small time steps, also lead to a large number of iterations and then 
to a large cost. In plasma physics, even if binary collisions are not the main effect compared to the mean field interaction, 
they are not negligible and have been considered in the non homogeneous case using Eulerian methods in 
\cite{filbet_pareschi, cf, gamba}. Collisions are currently taken into account for example in gyrokinetic simulations but 
in this context, a more simple choice to treat collisions consistently is to use the Fokker-Planck type operators which leads to 
a second order parabolic equation (see \cite{hammet, gene, multi_gene, gysela_coll}). As such, they have 
a moderate levels of stiffness, their spectrum belongs to the negative half space and is contained in a strip around the negative real axis.

For such collision operators, the use of implicit methods can be a strong obstacle due to cost and memory aspects 
but explicit methods may lead to very stringent stability constraints that also increase the computational cost. 
Hence, stabilized Runge-Kutta methods turn out to be an interesting alternative since in addition 
the conservation properties of the discretized collision operator are simply transferred to the fully discrete case. 
{Let us quote \cite{ulbl, angus} in which the authors use a suitable finite volume strategy to ensure the conservations 
	coupled with standard explicit Runge-Kutta method.} Let us also mention that stabilized Runge-Kutta methods 
have been already used in \cite{filbet_pareschi} for Vlasov-Poisson-Landau 
and in \cite{gene, multi_gene} for collisional gyrokinetic case. However in these works, 
the time integrator was not the main point and is therefore briefly discussed and one of the objectives of this work 
is to explain why these time integrators are a viable way for the weakly collisional plasma simulations. 

The rest of the paper is organized as follows: first the velocity discretization of the Fokker-Planck operator is 
presented ; second, the explicit stabilized Runge-Kutta methods are discussed and in a third part, the coupling between 
the Vlasov and collision part is presented. Finally, a bunch of numerical results are presented in different configurations 
to highlight the advantages of the approach. 

\section{Approximation of the Fokker-Planck operator}
\label{qf}
Here, we recall the Fokker-Planck collision operator and its properties (conservation and H-theorem) 
at the continuous level. Then, we present the velocity discretization and the properties at the semi-discrete level 
(discrete in velocity and continuous in time). Extensions to higher order and higher dimensions are provided.

The collision operator $Q(f)$ can be rewritten in the standard form, $\log$ form or $L^2$ form 
\begin{equation}
\label{reformulation_Q}
Q(f) = \nabla_v \cdot ((v-u_f) f +T_f \nabla_v f) = \nabla_v \cdot \Big(T_f f \nabla_v \log\Big(\frac{f}{M_f}\Big)\Big) = \nabla_v \cdot \Big(TM_f \nabla_v \Big(\frac{f}{M_f}\Big)\Big),   
\end{equation}
with $n_f =\int f dv, n_fu_f =\int vf dv, d_v n_f T_f=\int |v-u_f|^2 f dv$ and $M_f$ denotes the Maxwellian associated to $f$  
$$
M_f(v) = \frac{1}{(2\pi T_f)^{d/2}} \exp\Big(-\frac{|v-u_f|^2}{2T_f}\Big). 
$$
The operator $Q(f)$ enjoys the following properties 
$$
\int (1, v, |v|^2) Q(f) dv = 0, \;\; \int \log(f) Q(f) dv \leq 0, \;\; \int Q(f)\frac{f}{M_f}  dv \leq 0,  
$$
which is important to preserve at the numerical level. However, since the Fokker-Planck operator has been widely studied, 
there exists some suitable discretizations in the literature 
which ensure some of these properties. In particular, in \cite{dellacherie}, the authors made the link between the different formulations 
\eqref{reformulation_Q} of the Fokker-Planck operator (standard form, $\log$ form or $L^2$ form in \eqref{reformulation_Q}). 
According to the chosen form, the discretization will enjoy conservation properties or entropy decreasing property. 
In \cite{dellacherie}, the entropic averaging is proposed to make the link between the standard and $\log$ forms. 
However, using the $\log$ form may impose strong stability constraints on the numerical scheme to ensure the positivity of 
the solution. Thus, we follow the same path as in \cite{dellacherie} 
to make the link between the standard form (well suited for the conservations) and the $L^2$ form (well adapted for the entropy). 
{Let us mention some related works \cite{ulbl, angus} in which a second order finite volume method is used with a suitable 
	corrections to ensure conservations. }

\subsection{Second order numerical scheme}
To present the numerical scheme, we restrict ourselves to the case $d=1$ with $v_j= j\Delta v$ ($j\in\mathbb{Z}$ and $\Delta v>0$ the velocity mesh). 
We also introduce the  discrete unknown $f_j\approx f(v_j)$. 
We look for a discrete function $\tilde{M}_{j+1/2}$ (which is supposed to approximate a Maxwellian) such that the discretization 
of the standard form $\partial_v ((v-u_f) f +T_f \partial_v f)$ and the discretization of the $L^2$ form 
$\partial_v  \Big(T_f M_f \partial_v \Big(\frac{f}{M_f}\Big)\Big)$ are equivalent. 
For the standard form, we consider the following discretization  
\begin{equation}
\label{form1}
Q_j = \frac{1}{\Delta v} \Big[f_{j+1/2} (v_{j+1/2}-\tilde{u}) - f_{j-1/2} (v_{j-1/2}-\tilde{u})  \Big] + \frac{\tilde{T}}{\Delta v^2}(f_{j+1}-2f_j+f_{j-1}), 
\end{equation}
with 
\begin{equation}
\label{average_f}
f_{j+1/2} = \frac{f_{j+1} + f_j}{2}, \;\; \mbox{ and }  \;\; v_{j+1/2} = \frac{v_j+v_{j+1}}{2}, 
\end{equation}
whereas for the $L^2$ form, we consider the following discretization  
\begin{equation}
\label{form2}
Q_j =\frac{\tilde{T}}{\Delta v^2} \left\{ \tilde{M}_{j+1/2}\Big[\Big(\frac{f}{M}\Big)_{j+1} -\Big(\frac{f}{M}\Big)_{j} \Big] -\tilde{M}_{j-1/2}\Big[\Big(\frac{f}{M}\Big)_{j} -\Big(\frac{f}{M}\Big)_{j-1} \Big]  \right\}. 
\end{equation}
We need now to define the different quantities involved in the numerical schemes: first, for the macroscopic quantities $\tilde{u}$ and $\tilde{T}$,  
we choose  (following \cite{dellacherie}) 
\begin{equation}
\label{def_UT}
{n} = \sum_j f_{j} \Delta v, \;\;\;\; {n}\tilde{u} = \sum_j v_{j+1/2} f_{j+1/2} \Delta v \;\; \mbox{ and }  \; n \tilde{T} = \sum_j (v_{j+1/2} -\tilde{u})^2 f_{j+1/2} \Delta v. 
\end{equation}
and second the discrete Maxwellian $M_j$ is defined as 
\begin{equation}
\label{def_maxw}
M_j = M(v_j) = \frac{{n}}{\sqrt{2\pi\tilde{T}}} \exp(-|v_{j} -\tilde{u}]^2/ (2\tilde{T})).  
\end{equation}
As shown in \cite{dellacherie} (the calculation are recalled in \ref{proof_prop})), the first formulation \eqref{form1} enables to prove easily the conservations of mass, momentum and energy 
whereas the second one \eqref{form2}  enables to prove easily the entropy inequality. Hence, the goal 
is to find an expression of $\tilde{M}_{j+1/2}$ such that the two expressions are equivalent. 

We introduce the flux 
\begin{equation}
\label{flux_form1}
{\cal F}_{j+1/2} =  f_{j+1/2} (v_{j+1/2}-\tilde{u}) +\frac{\tilde{T}}{\Delta v} (f_{j+1} - f_j), 
\end{equation}
so that  \eqref{form1} can be expressed as $Q_j = ({\cal F}_{j+1/2}- {\cal F}_{j-1/2})/\Delta v$. Similarly, 
defining 
\begin{equation}
\label{flux_form2}
{\cal G}_{j+1/2} = \frac{\tilde{T}}{\Delta v}  \tilde{M}_{j+1/2}\Big[\Big(\frac{f}{M}\Big)_{j+1} -\Big(\frac{f}{M}\Big)_{j} \Big], 
\end{equation}
enables to express \eqref{form2} as $Q_j=({\cal G}_{j+1/2}  - {\cal G}_{j-1/2} )/\Delta v$. 
Hence, we identify  ${\cal F}_{j+1/2}$ and ${\cal G}_{j+1/2}$ to define $\tilde{M}_{j+1/2}$   
\begin{eqnarray*}
	{\cal G}_{j+1/2} \equiv \frac{\tilde{T}}{\Delta v}  \tilde{M}_{j+1/2}\Big(\frac{f_{j+1} M_j - f_jM_{j+1} }{M_j M_{j+1}}\Big)  \!\!&=&\!\!    f_{j+1/2} (v_{j+1/2}-\tilde{u}) +\frac{\tilde{T}}{\Delta v} (f_{j+1} - f_j)\equiv {\cal F}_{j+1/2}
\end{eqnarray*}
which enables to define $\tilde{M}_{j+1/2}$ as 
\begin{equation}
\label{def_M}
\tilde{M}_{j+1/2} =  \frac{\Delta v}{\tilde{T}}  \frac{M_jM_{j+1} {f}_{j+1/2} (v_{j+1/2} - \tilde{u})}{f_{j+1}M_j - f_jM_{j+1}} + \frac{M_j M_{j+1} (f_{j+1} - f_j)}{f_{j+1}M_j - f_jM_{j+1}}.  
\end{equation}
Hence, the standard discretization \eqref{form1} is equivalent to the discretization \eqref{form2} where $\tilde{M}_{j+1/2}$ is defined by \eqref{def_M} 
which enables us to prove the following Proposition.  
\begin{proposition}
	\label{prop_conservation}
	The discretization \eqref{form1}-\eqref{average_f}-\eqref{def_UT} of the Fokker-Planck operator $Q$ given by \eqref{reformulation_Q} satisfies 
	$$
	\sum_{j\in \mathbb{Z}} Q_j 
	\left( 
	\begin{array}{lll}
	1\\ v_j\\ v_j^2/2
	\end{array}
	\right) \Delta v = 
	\left( 
	\begin{array}{lll}
	0\\ 0\\ 0 
	\end{array}
	\right)  \;\;\;\; \mbox{ and } \;\;\;\; \sum_{j\in \mathbb{Z}} Q_j  \frac{f_j}{M_j }\Delta v \leq 0,  
	$$
	which means that the discrete entropy ${\cal E}(t) = \Delta v \sum_{j\in \mathbb{Z}} (f^2_j)/M_j$ satisfies $\frac{d}{dt} {\cal E}(t) \leq 0$.  
\end{proposition}
The proof is given in \ref{proof_prop}.

\subsection{Fourth order numerical scheme}
Based on the discretization \eqref{form1} of $Q$, the goal of this part is to derive a higher order velocity discretization for $Q$. 
First, integrating the standard form of $Q$ on $[v_{j-1/2}, v_{j+1/2}]$ leads to 
\begin{eqnarray*}
	\bar{Q}_j &\equiv & \frac{1}{\Delta v}\int_{v_{j-1/2}}^{v_{j+1/2}} Q(f)(v) dv \nonumber\\
	&=& \frac{1}{\Delta v}\Big((v_{j+1/2} -u_f) f(v_{j+1/2}) - (v_{j-1/2}-u_f)  f(v_{j-1/2}) + T_f (\partial_vf(v_{j+1/2}) -\partial_vf (v_{j-1/2}))\Big), 
\end{eqnarray*} 
where the notations $u_f, T_f$ are given by \eqref{eq:Tu}.  
Using the following fourth order finite difference formula 
\begin{eqnarray}
\label{midpoint}
f(v_{j+1/2}) &\approx& \breve{f}_{j+1/2} =-\frac{1}{16} f_{j-1} +\frac{9}{16} f_j +\frac{9}{16} f_{j+1}  -\frac{1}{16} f_{j+2}, \\
(\partial_v f)(v_{j+1/2}) &\approx& \frac{f_{j-1} -27 f_{j} +27 f_{j+1}  - f_{j+2}}{24\Delta v}, \nonumber
\end{eqnarray} 
we get the following fourth order finite volume approximation $\bar{Q}^{[4]}_j$ of $\bar{Q}_j$ 
\begin{equation}
\label{fourthorder_vol}
\bar{Q}^{[4]}_j=  \frac{(v_{j+1/2}\!-\!\breve{u})\breve{f}_{j+1/2}\!  -\! (v_{j-1/2}\!-\!\breve{u})\breve{f}_{j-1/2}   }{\Delta v} + \breve{T}\Big(\frac{-f_{j-2} + 28 f_{j-1} - 54 f_j +28f_{j+1}-f_{j+2}}{24\Delta v^2}\Big), 
\end{equation}
with $\breve{f}_{j+1/2}$ given by \eqref{midpoint} and the definitions of  $\breve{u}$ and $\breve{T}$ have to be adapted in this case: 
\begin{equation}
\label{def_UT_fourth}
n =\Delta v\sum_j f_j, \;\;\;\; \breve{u} =\Delta v \sum_j v_{j+1/2}\breve{f}_{j+1/2}, \;\; \mbox{ and } \;\;  n\breve{T} =\Delta v \sum_j (v_{j+1/2} - \breve{u})^2 \breve{f}_{j+1/2}. 
\end{equation}
To get a pointwise approximation (ie of $Q(f)(v_j)$), we will use the following relation between cell average and 
the pointwize value. For all smooth function $g$, we get   
\begin{equation}
\label{vol_to_point}
\bar{g}_j =  \frac{1}{\Delta v}\int_{v_{j-1/2}}^{v_{j+1/2}} g(v) dv = g(v_j) + \frac{\Delta v^2}{24}g''(v_j) + {\cal O}(\Delta v^4), 
\end{equation}
from which we deduce the following relation 
$$
g(v_j) = \bar{g}_j - \frac{\Delta v^2}{24}\frac{g(v_{j+1})-2g(v_j) + g(v_{j-1}) }{\Delta v^2} + {\cal O}(\Delta v^4).  
$$
We then deduce a fourth order approximation $Q^{[4]}_j $ of $Q(f)(v_j)$ 
gathering the previous ingredients 
\begin{equation}
\label{fourthorder}
Q^{[4]}_j = \bar{Q}^{[4]}_j - \frac{1}{24}\Big[Q_{j+1}-Q_{j} + Q_{j-1}  \Big] \;\; \mbox{ with } Q_j \mbox{ given by \eqref{form1}-\eqref{average_f}-\eqref{def_UT}}, 
\end{equation}
where $\bar{Q}^{[4]}_j$ is given by \eqref{fourthorder_vol}.  

In the following Proposition, we gather the properties of this new discretization. 
\begin{proposition}
	\label{fourth_order_prop}
	The discretization \eqref{fourthorder}-\eqref{midpoint}-\eqref{fourthorder_vol}-\eqref{def_UT_fourth} of the Fokker-Planck operator $Q$  given by \eqref{reformulation_Q} satisfies 
	$$
	Q(f)(v_j)- Q^{[4]}_j ={\cal O}(\Delta v^4), 
	$$
	and 
	$$
	\sum_j Q^{[4]}_j \left( 
	\begin{array}{lll}
	1\\ v_j\\ v_j^2/2
	\end{array}
	\right)  \Delta v = 
	\left( 
	\begin{array}{lll}
	0\\ 0\\ 0 
	\end{array}
	\right).  
	$$
\end{proposition}
\begin{proof}
	First, we have from \eqref{vol_to_point} 
	\begin{eqnarray*}
		\bar{Q}_j &=& Q(f)(v_j) + \frac{\Delta v^2}{24}(\partial_v^2 Q(f))(v_j) + {\cal O}(\Delta v^4) \nonumber\\
		&=& Q(f)(v_j) + \frac{1}{24}(Q(f)(v_{j+1})-2Q(f)(v_j)+Q(f)(v_{j-1})) + {\cal O}(\Delta v^4),  
	\end{eqnarray*}
	and we also have $\bar{Q}_j = \bar{Q}^{[4]}_j + {\cal O}(\Delta v^4)$, from which we deduce 
	$$
	Q(f)(v_j) = \bar{Q}^{[4]}_j -\frac{1}{24}(Q(f)(v_{j+1})-2Q(f)(v_j)+Q(f)(v_{j-1})) + {\cal O}(\Delta v^4).   
	$$
	It remains to check that $Q(f)(v_{j+1})-2Q(f)(v_j)+Q(f)(v_{j-1}) - (Q_{j+1}-Q_{j} + Q_{j-1}) = {\cal O}(\Delta v^4)$, 
	where $Q_{j}$ is given by \eqref{form1}-\eqref{average_f}-\eqref{def_UT}.  
	First, we look at the consistency error $Q(v_j) -   Q_{j}$ by performing Taylor expansions of $f(v_{j\pm 1})$  in $Q_{j}$ 
	\begin{eqnarray*}
		Q_j 
		&=& Q(f)(v_j) +  (u_f-\tilde{u}) (\partial_v f)(v_j)+ (T_f-\tilde{T} ) (\partial^2_v f)(v_j)  \nonumber\\
		&& +\Delta v^2 \Big(\frac{1}{4}(\partial^2_v f)(v_j) + \frac{1}{24}(v_j -\tilde{u}) \partial_v^3 f(v_j)+ \frac{\tilde{T}}{24}(\partial_v^4 f)(v_j) \Big) + {\cal O}(\Delta v^4). 
	\end{eqnarray*}
	Thus, we have 
	\begin{eqnarray*}
		Q_{j+1}-2 Q_{j}+ Q_{j-1} &=& Q(f)(v_{j+1})-2Q(f)(v_j)+Q(f)(v_{j-1}) \nonumber\\
		&&+(u_f-\tilde{u}) \Delta v^2 (\partial_v^3 f)(v_j) +(T_f-\tilde{T} )\Delta v^2 (\partial^4_v f)(v_j) + {\cal O}(\Delta v^4) \nonumber\\
		&=& \Delta v^2 (\partial_v^2 Q(f))(v_j) +{\cal O}(\Delta v^4), 
	\end{eqnarray*}
	since $u_f-\tilde{u}=T_f-\tilde{T}={\cal O}(\Delta v^2)$ using second order quadrature. We deduce 
	\begin{eqnarray*}
		Q(f)(v_j) - Q^{[4]}_j &=& Q(f)(v_j) - \Big[ \bar{Q}^{[4]}_j -\frac{1}{24}(Q_{j+1}-2Q_{j} + Q_{j-1})\Big]   \nonumber\\
		&=& Q(f)(v_j) -  \bar{Q}^{[4]}_j + \frac{\Delta v^2}{24}\partial_v^2 Q(v_{j})+ {\cal O}(\Delta v^4) \nonumber\\
		&=& Q(f)(v_j) -  \Big[\bar{Q}_j - \frac{\Delta v^2}{24}\partial_v^2 Q(v_{j})\Big]+ {\cal O}(\Delta v^4) \nonumber\\
		&=& {\cal O}(\Delta v^4).  
	\end{eqnarray*}

	Next, we prove the conservations properties. We first prove that the conservations hold true for $\bar{Q}^{[4]}_j$. 
	The flux formulation implies directly mass conservation. Next, we introduce the notation $(D f)_j = (-f_{j-2} + 28 f_{j-1} - 54 f_j +28f_{j+1}-f_{j+2})/(24\Delta v^2)$ which is a fourth order 
	approximation of $(\partial_v^2 f)(v_j)$. The momentum conservation follows from the property $\sum_j  v_j (Df)_j \Delta v = 0$  
	which allows to follow the lines  of the proof of Proposition \ref{prop_conservation}.  
	For the energy, we first check  $\Delta v\sum_j (v^2_j/2) (Df)_j = n$ by two discrete integration by parts. 
	Thus, we follow the lines of the proof of Proposition \ref{prop_conservation} to get 
	$\sum_j \bar{Q}^{[4]}_j \Big(1, v_j, v^2_j\Big) \Delta v = (0, 0, 0)$. \\  
	Indeed, we have directly  that $\sum_j Q^{[2]}_j \Big(1, v_j, v^2_j\Big) \Delta v = (0, 0, 0)$. 
	It remains to check that $\sum_j [{Q}_{j+ 1} - 2Q_j+Q_{j-1}] \Big(1, v_j, v^2_j\Big) \Delta v = (0, 0, 0)$. The first component is 
	obvious using Proposition \ref{prop_conservation}. For the momentum, we have  
	\begin{eqnarray*}
		\sum_j {Q}_{j\pm 1} v_j \Delta v &=&   \sum_j {Q}_{j} v_{j\mp 1} \Delta v =  \sum_j {Q}_{j} (v_{j} \mp \Delta v) \Delta v \nonumber\\
		&=&  \sum_j {Q}_{j} v_{j} \Delta v  \mp  \sum_j {Q}_{j} \Delta v^2 =0, 
	\end{eqnarray*}
	again from Proposition \ref{prop_conservation}. 
	For the energy, 	we proceed similarly to get $\sum_j {Q}_{j\pm 1} v^2_j \Delta v = 0$ which enable to conclude the proof. 
\end{proof}

\begin{remark}
	For the entropy, it is possible to follow the methodology used for the second order approach to construct \eqref{def_M}.  
	Indeed, we can derive the fluxes from  \eqref{fourthorder} 
	\begin{eqnarray*}
		{\cal F}_{j+1/2} &=& (v_{j+1/2} - \breve{u}) \breve{f}_{j+1/2} \!+\! \frac{\breve{T}}{24\Delta v}(-f_{j+2} + 27 f_{j+1}- 27 f_{j} + f_{j-1}) \nonumber\\
		&-& \frac{1}{24}({\cal F}^{[2]}_{j+3/2}-2{\cal F}^{[2]}_{j+1/2}+{\cal F}^{[2]}_{j-1/2})   
	\end{eqnarray*}
	with ${\cal F}^{[2]}_{j+1/2}$ the flux of the second order version given by \eqref{flux_form1}, $n=\sum_j f_j \Delta v$ and 
	\begin{eqnarray*}
		n \breve{u} &=& \Delta v\sum_j v_{j+1/2}\breve{f}_{j+1/2}, \;\;\;  n\breve{T} = \Delta v\sum_j (v_{j+1/2} - \breve{u})^2 \breve{f}_{j+1/2},\nonumber\\  
		n\tilde{u} &=& \Delta v\sum_j v_{j+1/2}{f}_{j+1/2}, \;\;\;   n\tilde{T} = \Delta v\sum_j (v_{j+1/2} - \tilde{u})^2 {f}_{j+1/2}, 
	\end{eqnarray*}
	with $\breve{f}_{j+1/2}$ is given by \eqref{midpoint} and $f_{j+1/2} = (f_{j+1} + f_j)/2$. 
	On the other side, we derive the fourth order fluxes for the $L^2$ form (with $g_j=(f/M)_j$) 
	\begin{eqnarray*}
		{\cal G}_{j+1/2} &=& \frac{\breve{T} \breve{M}_{j+1/2}}{24 \Delta v}  (g_{j-1} - 27 g_{j} +  27 g_{j+1} - g_{j+2})-  \frac{1}{24} ({\cal G}^{[2]}_{j+3/2}-2{\cal G}^{[2]}_{j+1/2}+{\cal G}^{[2]}_{j-1/2})  
	\end{eqnarray*}
	with $\breve{M}_{j+1/2}$ is to be computed and ${\cal G}^{[2]}_{j+1/2}$ the flux of the second order version given by \eqref{flux_form2}. 
	Identifying the fluxes ${\cal F}_{j+1/2}$ and ${\cal G}_{j+1/2}$ enables to derive $\breve{M}_{j+1/2}$:  
	\begin{eqnarray*}
		\breve{M}_{j+1/2} &=& \frac{1}{\alpha_j}\Big[{\cal F}_{j+1/2} +\frac{1}{24} \Big({\cal G}^{[2]}_{j+3/2}-2{\cal G}^{[2]}_{j+1/2}+{\cal G}^{[2]}_{j-1/2} \Big)\Big], \nonumber\\
		\mbox{with } \alpha_j &=& \frac{\breve{T}}{24\Delta v} (g_{j-1} - 27 g_{j} +  27 g_{j+1} - g_{j+2}), \nonumber\\
		&=& \frac{\breve{T} (f_{j-1}M_{j}M_{j+1}M_{j+2} -27 f_j M_{j-1}M_{j+1}M_{j+2}  + 27 f_{j+1}M_{j-1}M_{j}M_{j+2}  - f_{j+2}M_{j-1}M_{j}M_{j+1})  }{24 M_{j-1}M_{j}M_{j+1}M_{j+2} \Delta v}
	\end{eqnarray*}
	ensuring that the two forms are equivalent from a discrete point of view. Some calculations similar to the ones performed in the end of \ref{proof_prop} enable to prove that $\breve{M}_{j+1/2} =M(v_{j+1/2}) + {\cal O}(\Delta v^4)$.

\end{remark}

\section{Explicit stabilized Runge-Kutta methods}  
The space dicretization of a diffusion dominated advection-diffusion PDE leads to a high dimensional system of ODEs of the form 
\begin{equation}\label{eq:ode}
\dot y = F(y), \quad y(0)=y_0,
\end{equation}
where the eigenvalues of the Jacobian of the vector field $F$ are located in a narrow strip around the negative real axis. In order to study the stability of Runge-Kutta methods, it is common to consider the linear equation (see \cite[Chap. IV.2]{HaW96})
\begin{equation*}\label{eq:linearode}
\dot y = \lambda y, \quad y(0)=y_0,
\end{equation*}
where $\lambda \in \mathbb{C}$. Applying a RK method with time step $\Delta t$ to the above linear ODE leads to $y^n=R(\lambda \Delta t)^ny_0$ where $R(z)$ is called the stability function, and then the stability domain of the method is defined by
\begin{equation*}
\mathcal{S}=\left\{z\in\mathbb{C}\,;\,|R(z)|\leq 1\right\}.
\end{equation*}

For stiff equations, the fact that the stability function of explicit methods is always a polynomial, which means bounded stability region, causes restrictions on the step size which turn out to be severe in the case of standard explicit schemes such as the forward Euler method. Implicit schemes are a good alternative for low dimensional problems, but their cost becomes prohibitive when solving high dimensional problems since they require the solution of high dimensional linear or/and nonlinear systems.

Explicit stabilized RKC methods (for Runge--Kutta--Chebyshev) have the advantages of both worlds. They do not require the solution of linear/nonlinear systems and, at the same time, they allow large step sizes thanks to their extended stability domains over the negative real axis. The literature of these methods is rich, and many schemes of this types were introduced to solve different types of stiff problems such as advection--diffusion--reaction equations {\cite{A22,AV13,zbinden}}, stochastic differential equations {\cite{AAV18,AbL08,AVZ13}}, and optimal control problems \cite{AV21}. 
In particular, for the advection--diffusion problems, recent papers improved the stability domain of stabilized Runge-Kutta method along the imaginary axis to take into account relatively large Peclet number \cite{AV13, A22}. Of course, when advection strongly dominates the diffusion part (very large Peclet number), such methods are not appropriate and standard high order Runge-Kutta methods may be preferred. 
For a complete review on explicit stabilized RKC methods for stiff ODEs, we refer the reader to the review \cite{Abdrev}.

{
	For numerical simulations we will consider a bounded velocity domain of the form $[-v_{\max},v_{\max}]$ where $v_{\max}\in \R$ large enough, which is then discretized using $N_v+1$ points $v_j=-v_{\max}+j\Delta v,\,j=0,\dots,N_v$, where $\Delta v=2v_{\max}/N_v$, and $N_v$ is a positive integer. In Figure \ref{fig:ev}, we plot the eigenvalues of the discrete $Q$ for different values of $v_{\max}$, the larger $v_{\max}$ is, the greater the advection term and thus the larger the imaginary parts of the eigenvalues.
}

\begin{figure}[tb]
	\centering
	\begin{subfigure}[t]{0.32\textwidth}
		\includegraphics[width=\linewidth]{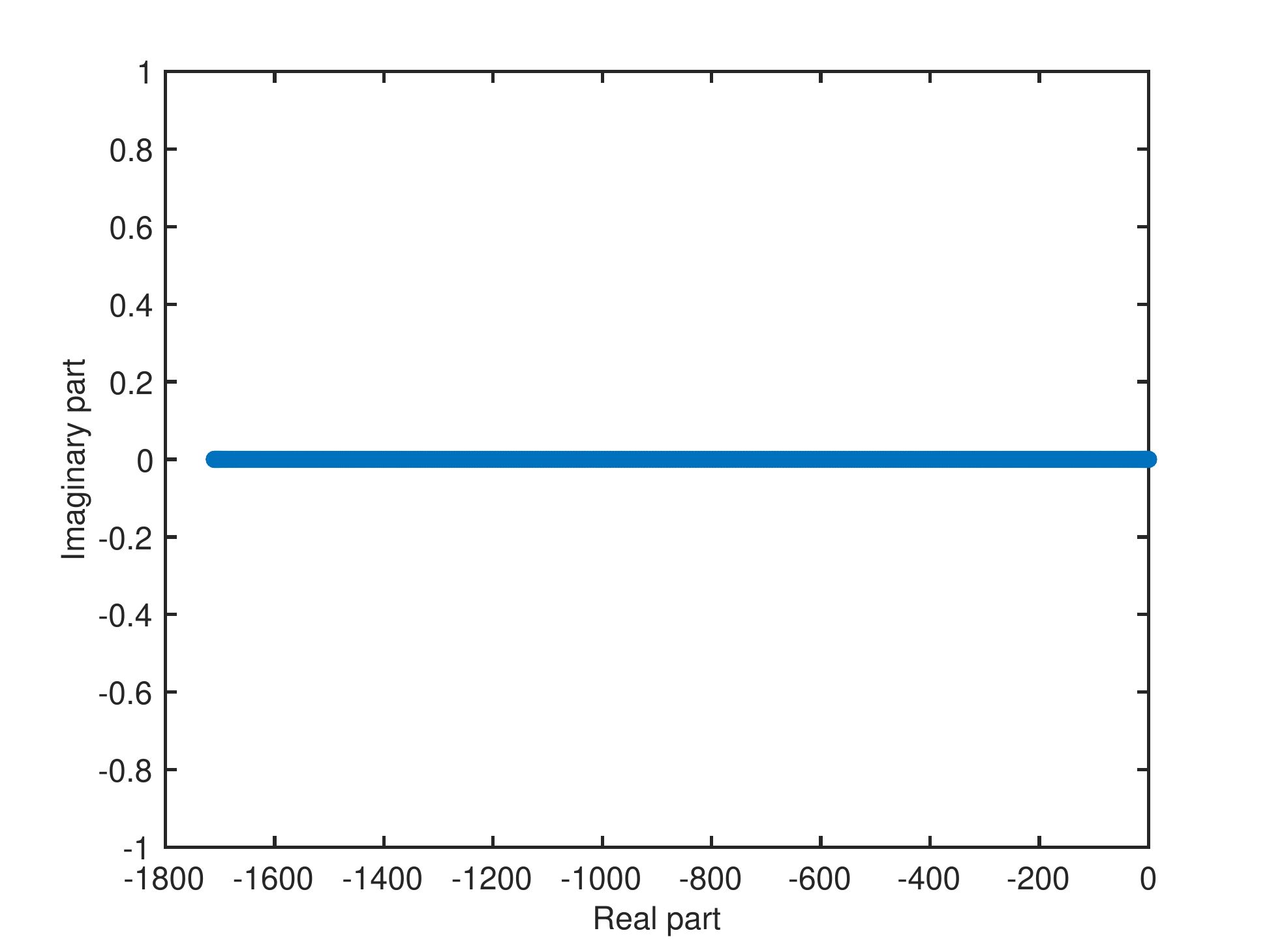}
		\caption{$v_{\max}=12$.}
	\end{subfigure}
	\begin{subfigure}[t]{0.32\linewidth}
		\includegraphics[width=\linewidth]{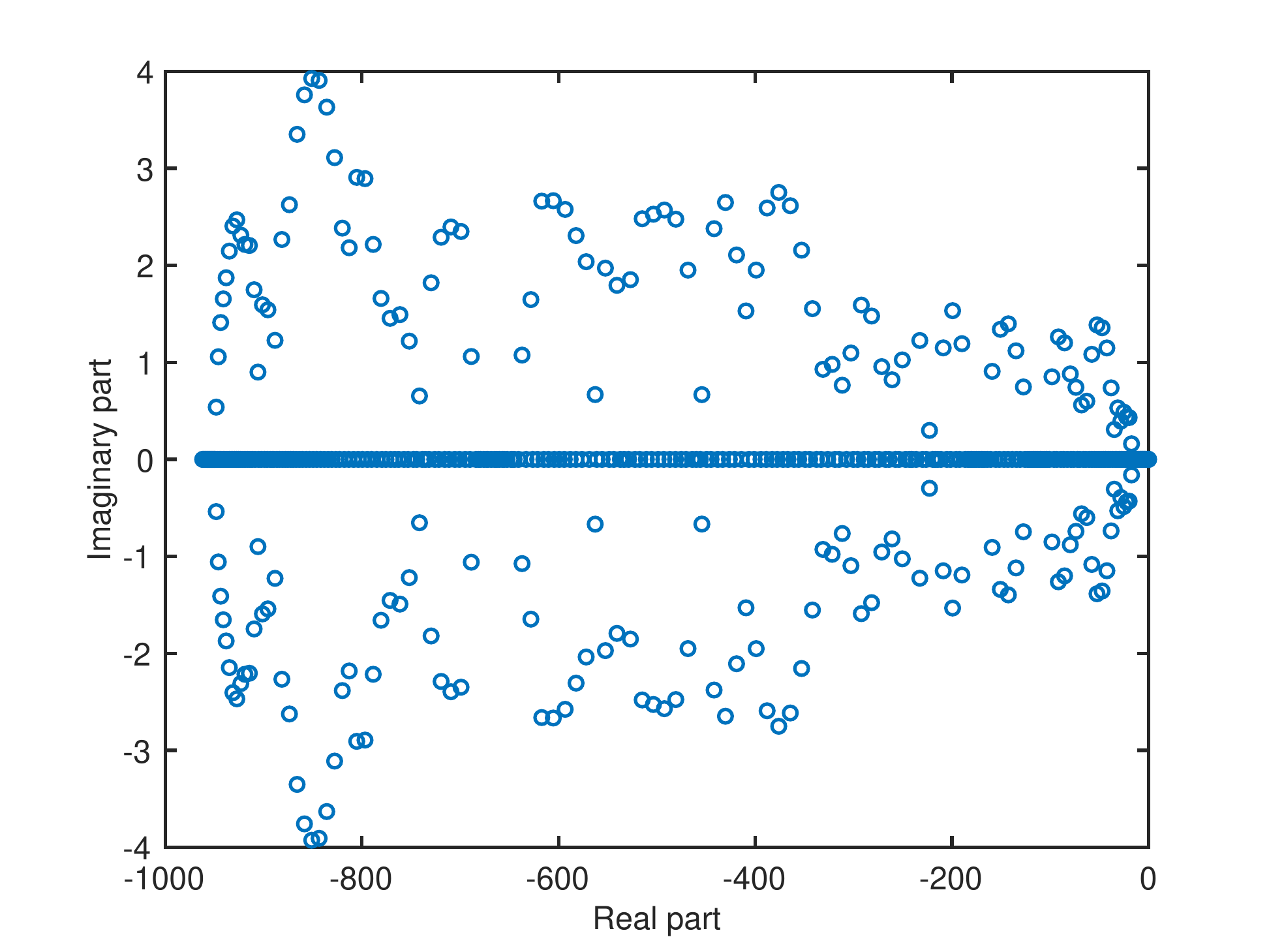}
		\caption{$v_{\max}=16$.}
	\end{subfigure}
	\begin{subfigure}[t]{0.32\linewidth}
		\includegraphics[width=\linewidth]{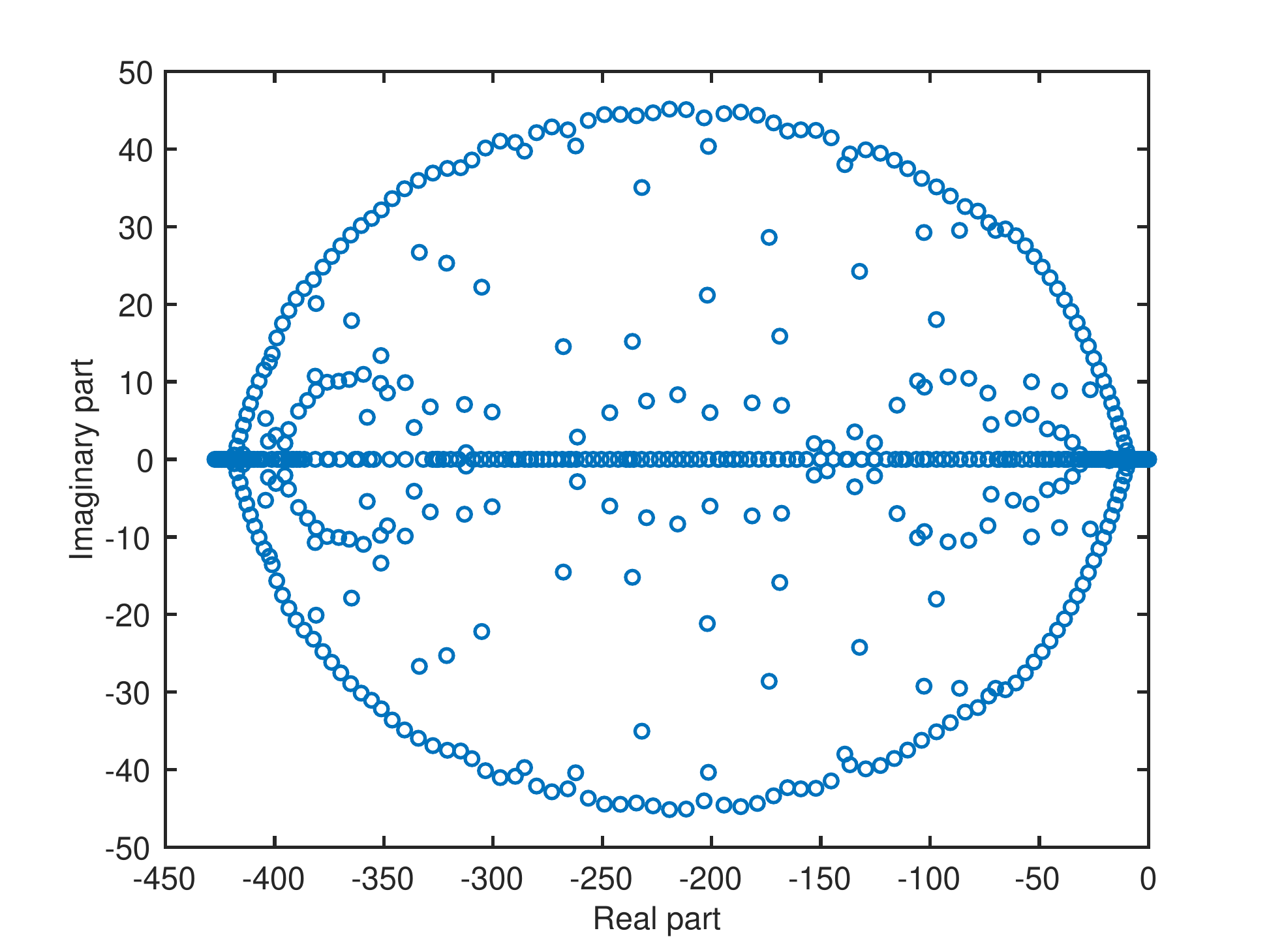}
		\caption{$v_{\max}=24$.}
	\end{subfigure}
	\caption{
		Eigenvalues of $Q$ for $N_v=512$, $T=1.88$, $u=0$, $\nu=0.5$, and different values of $v_{\max}$.
		\label{fig:ev}}
\end{figure}

\subsection{Solving the velocity part using RKC methods}
Here, we detail how to solve numerically $\partial_t f = \nu Q(f)$ or $\partial_t f + E\partial_v f = \nu Q(f)$ depending on the chosen splitting strategy (see Section \ref{sec:coupling_vlasov}). For the second equation, we combine the transport part $E\partial_v f$ 
with collisional part as 
$$
\partial_t f =\nu \tilde{Q}(f) := \nu\partial_v((v-u-\frac{E}{\nu})f + T\partial_v f). 
$$
We will present the algorithm for $\partial_t f=Q(f)$, which applies in the same way on $\partial_t f={\tilde Q(f)}$. {We will omit the coefficient $\nu$ for simplicity, and to avoid confusion with the coefficients $\nu_\ell$ of the RKC method}. After a velocity discretization of $Q(f)$ following the methods presented in Section \ref{qf}, 
it remains to integrate in time. Motivated by the fact that the eigenvalues of the Jacobian of $Q$ are localized in the negative half space (see Figure \ref{fig:ev}), we propose here to use a stabilized Runge--Kutta method which enables to satisfy 
several properties. Indeed, since it is an explicit method, it avoids a potentially costly implicit solver, 
moreover, the discrete properties  of $Q$ can be transferred easily to the fully discrete case and of course 
it is not constrained by the severe parabolic CFL condition and  enables to use large time steps. 

Then, once discretized in velocity, we obtain 
\begin{equation}\label{eq:disode}
\frac{d f_j(t)}{dt}  = Q_j, \;\;\;\;\;\; \mbox{ {with} } {Q_j  \approx  Q(f)(v_j), }
\end{equation}
which is amenable for stabilized RKC methods. {Before presenting the RKC schemes, let us quickly recall the definition of first kind Chebyshev polynomials which are the key ingredient in constructing this type of stabilized methods. The first kind Chebyshev polynomials are the unique polynomials that satisfy, for all complex number $z$, the following recurrence
	\begin{equation}
	T_0(z)=1,\quad T_1(z)=z, \quad T_j(z)=2zT_{j-1}(z)-T_{j-2}(z),\,\,\,\, j\geq2.
	\end{equation}
}

The first order RKC method (RKC1) applied to \eqref{eq:disode} reads
\begin{eqnarray}
K^{(0)} &=& f^n_j \nonumber\\ 
K^{(1)} &=& K^{(0)}  + \mu_1  \Delta t Q^n_j, \nonumber\\
\mbox{for} \!\!\!\!&&\!\!\!\! \ell=2, \dots, s\nonumber\\
K^{(\ell)} &= &  \kappa_\ell K^{(\ell-2)} + \nu_\ell K^{(\ell-1)} + \mu_\ell \Delta t Q^{(\ell-1)}_j, \;\;\;  \nonumber\\
\label{rkc_algo}
{f}^{n+1}_j &=&  K^{(s)},
\end{eqnarray}
where,
\begin{equation}\label{eq:coeffscheb}
\mu_1=\frac{\omega_1}{\omega_0},\quad
\mu_\ell =\frac{2\omega_1T_{\ell-1}(\omega_0)}{T_\ell(\omega_0)},\quad 
\nu_\ell=\frac{2\omega_0T_{\ell-1}(\omega_0)}{T_\ell(\omega_0)},\quad \kappa_\ell=1-\nu_\ell,\quad \ell=2,\ldots,s.
\end{equation}
for $\omega_{0}=1+\frac{\eta}{s^{2}},$ $\omega_{1}=\frac{T_{s}(\omega_{0})}{T'_{s}(\omega_{0})}$, and $\eta=0.05$.
The second order RKC method (RKC2) reads
\begin{eqnarray*}
	K^{(0)} &=& f^n_j \\ 
	K^{(1)} &=& K^{(0)}  + \hat\mu_1  \Delta t Q^n_j, \\
	\mbox{for} \!\!\!\!&&\!\!\!\! \ell=2, \dots, s\\
	K^{(\ell)} &= & (1-\hat\kappa_\ell-\hat\nu_\ell) K^{(0)} + \hat\kappa_\ell K^{(\ell-2)} + \hat\nu_\ell K^{(\ell-1)} + 
	\hat\mu_\ell \Delta t Q^{(\ell-1)}_j - a_{\ell-1} \hat\mu_\ell \Delta t Q^n_j, \;\;\;  \\
	{f}^{n+1}_j &=&  K^{(s)},
\end{eqnarray*}
where
\begin{equation}\label{eq:coeffsrkc}
\hat\mu_\ell=\frac{2b_\ell\omega_2}{b_{\ell-1}},~~ 
\hat\nu_\ell=\frac{2b_\ell\omega_0}{b_{\ell-1}},~~ 
\hat\kappa_\ell=-\frac{b_\ell}{b_{\ell-2}},~~
b_\ell=\frac{T''_\ell(\omega_0)}{(T'_\ell(\omega_0)^2)},~~
a_\ell=1-b_\ell T_\ell(\omega_0),
\end{equation}
for $\omega_0=1+\frac{\eta}{s^2},\,
\omega_2=\frac{T'_s(\omega_0)}{T''_s(\omega_0)},$ and $\eta=0.15.$
In the above equations, $Q^{(\ell)}_j=Q(f^{(\ell)}_j)$, $s\in\mathbb{N}\backslash\{0,1\}$ is the number of stages of the method. The stability region of the RKC schemes contains a strip around the interval $[-C_\eta s^2,0]$ where, typically, $C_\eta=1.93$ for RKC1 and $C_\eta=0.65$ for RKC2. The dependence on $\eta$ will be clarified in the next subsection. For a given step size $\Delta t$, the number of stages $s$ is chosen such that the product $\lambda_{\max}\Delta t$ lies in the stability region, where $\lambda_{\max}$ is the eigenvalue with the maximum real part in absolute value. 

\begin{proposition}\label{prop:rkc_cons}
	Consider the homogeneous equation $\partial_t f=\nu Q(f)$ with $Q$ given by \eqref{reformulation_Q}. Suppose that we discretize $Q(f)$ in velocity using \eqref{form1}-\eqref{average_f}-\eqref{def_UT} (second order version) or using  \eqref{fourthorder}-\eqref{midpoint}-\eqref{fourthorder_vol}-\eqref{def_UT_fourth} (fourth order version), and {RKC1 or RKC2} in time. Let us denote by $f^n_j$ the approximation of $f(t^n,v_j)$, then the resulting numerical scheme preserves mass, momentum, and kinetic energy 
	$$
	\text{For all $n\in \mathbb{N}$},\quad \sum_{j}   \left( 
	\begin{array}{lll}
	1\\ v_j\\ v_j^2/2
	\end{array}
	\right)  f^{n}_j=\sum_{j} \left( 
	\begin{array}{lll}
	1\\ v_j\\ v_j^2/2
	\end{array}
	\right)  f^0_j. 
	$$
\end{proposition}

\begin{proof}
	The proof follows easily by induction on the internal stages of the RKC methods.
\end{proof}

\subsection{Damping}
The coefficients of the above methods depend on a parameter usually denoted $\eta$ and called \emph{damping parameter}. In the absence of damping, $(\eta=0)$, the stability region has singular points that correspond to the value $\pm1$ of the stability function. In order to make the schemes robust with respect to possible small perturbations of the eigenvalues, coming from advection terms for example, a positive value of $\eta$ is needed. For pure diffusion or diffusion dominated advection--diffusion operators, $\eta$ is usually fixed to $0.05$ for RKC1 and $0.15$ for RKC2 (hence the values $1.93$ and $0.65$ of $C_\eta$ respectively). However, increasing $\eta$ gives the stability region more width in the imaginary direction which gives the scheme the ability to handle larger advection terms with reasonable time step sizes. Figure \ref{fig:stabdom} illustrates this idea.
Note that this gain in the imaginary direction comes at the cost of reducing the constant $C_\eta$ and then a reduction in the length of the real negative interval included in the stability domain. A good estimation of this constant is $C_\eta=(1+\omega_0)/\omega_1$. 


\begin{figure}[tb]
	\centering
	\begin{subfigure}[t]{1\textwidth}
		\centering
		\includegraphics[width=0.9\linewidth]{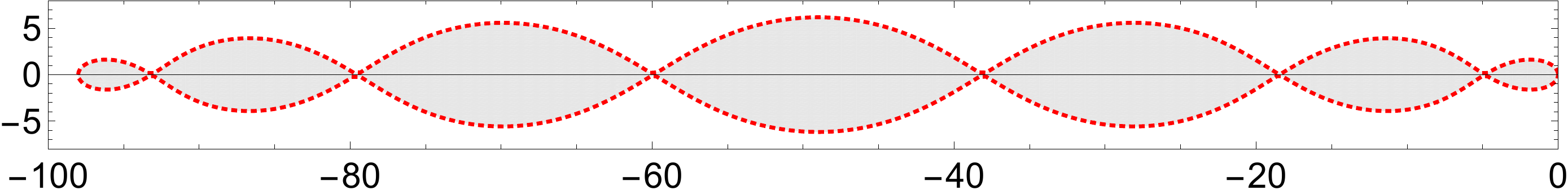}\\[2.ex]
		\includegraphics[width=0.9\linewidth]{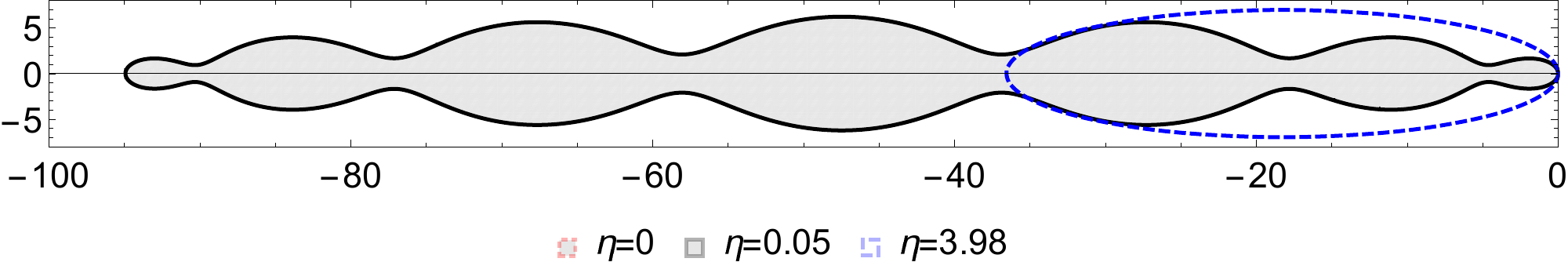}
		\caption{Complex stability domain for different damping parameters $\eta$.}
	\end{subfigure}
	\medskip
	\begin{subfigure}[t]{1\textwidth}
		\centering
		\includegraphics[width=0.9\linewidth]{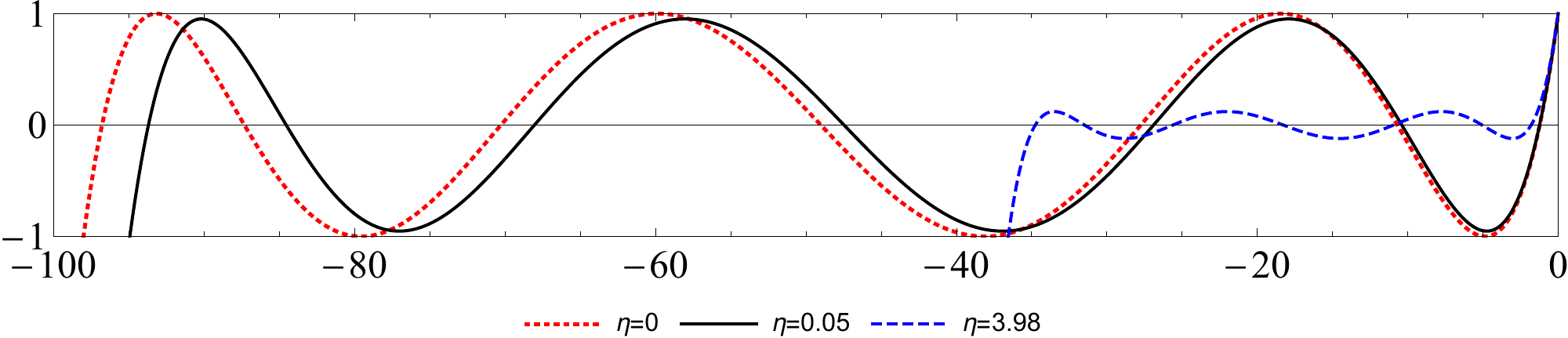}
		\caption{Corresponding stability functions.}
	\end{subfigure}
	\caption{
		Stability domains and stability functions of the Chebyshev method (RKC1) for $s=7$ and different damping values $\eta=0,0.05,3.98$.
		\label{fig:stabdom}}
\end{figure}

\subsection{Adaptive time stepping  strategy}
\label{adaptive}
We will use the automatic time step selection procedure adapted to the second order RKC method from \cite{SSV98}. The estimator used to estimate the local error of the method is the following
$$
Est_{n+1}=\frac1{15}[12(f^n-f^{n+1})
+6\Delta t(Q^n+Q^{n+1})].
$$
In order to reach a given accuracy, the adaptive time step selection is done in the spirit of the procedure explained in \cite{SSV98}. For reference, see also \cite[P. 167]{HNW93}. After choosing the suitable time step, the number of stages is chosen using the formula
\begin{equation*}
s := \left[\sqrt{\frac{\Delta t\lambda_{\max}+1.5}{C_\eta}}+0.5\right],
\end{equation*}
where the brackets mean rounding to the nearest integer. Note that the automatic time step control takes care as well of possible need to reduce the step size due to large advection during the integration process. This large effect of the advection term might appear when the maximum value of the velocity $v_{\max}$ is quite large.
\label{sec:est}

\section{Coupling with the Vlasov equation}
\label{sec:coupling_vlasov}
To solve the full Vlasov-Poisson-Fokker-Planck system, we use (Strang) splitting techniques between the different operators following two different strategies. 
The first strategy is classical {and relatively easy to implement}, while the second smartly manipulated some terms to exactly conserve the total energy.

\subsection{Classical splitting}
Even if a three steps splitting (between transport in space, velocity and collision) can be easily obtained, 
the velocity  transport part can be merged with the collision part so that the resulting splitting called SL-RKC is 
\begin{itemize}
	\item Solve the transport in space $\partial_t f + v\partial_x f = 0$ using a semi-Lagrangian method or Fourier technique.
	\item Solve $\partial_xE=1-\int_{\mathbb{R}}f dv$ using Fourier transform. 
	\item Solve $\partial_t f + E\partial_v f = \nu Q(f)$ using a stabilized Runge--Kutta--Chebyshev method (RKC).
\end{itemize}
The first two parts are classical and the last part can be rewritten as 
$$
\partial_t f = \nu \partial_v ((v-u_f - \frac{E}{\nu})f + T_f \partial_v f) = \nu\tilde{Q}(f), 
$$
which has the same structure as $Q(f)$ so that the numerical schemes presented in Section \ref{qf} can be applied. 

\subsection{Non-classical splitting}\label{sec:srr}
The second strategy is motivated by total energy (kinetic plus electric) conservation by using ideas from \cite{cf,HF01} and applied to 
Vlasov-Amp\`ere instead of Vlasov-Poisson. The algorithm called SL-RK2-RKC is 
\begin{itemize}
	\item Solve the transport in space $\partial_t f + v\partial_x f = 0$ using a semi-Lagrangian method or Fourier technique.
	\item Instead of solving a Poisson equation for the electric field $E$, solve the corresponding Amp\`ere equation $\partial_tE=-\int_{\R^{d_v}}vfdv$.
	\item Solve $\partial_t f + E\partial_v f = 0$ using a suitable second order  Runge--Kutta method and suitable approximation of $E$ that guarantees the exact conservation of total energy. Details will be given later in the current section.
	\item Solve $\partial_t f = \nu Q(f)$ using RKC.
\end{itemize}
{	The time step test is done after solving the collision part using RKC, and the integration is repeated from the beginning after a step is rejected.}

For this non-classical splitting, we will benefit from a smart idea introduced in the papers \cite{cf,HF01} of adapting the electric field approximation in order to construct a numerical scheme that preserves exactly the total energy of the system. 
Here, we will follow \cite{HF01} and use Amp\`ere equation 
for the electric field instead of Poisson equation which will be used only once at the beginning {of the simulation} to get an approximation of the initial value of $E$. In what follows, we will set $d_x=d_v=1$.

Let $D_xf$ be a (at least) consistent finite difference approximation of $\partial_xf$, and $D_vf$ be a (at least) second order finite differences approximation of $\partial_vf$ (centered or upwind). 
Let $f_{i,j}^n$ denote an approximation of $f(t^n, x_i, v_j)$ with $t^n=n\Delta t, x_i=i\Delta x, v_j=j\Delta v$ with $\Delta t, \Delta x, \Delta v$ 
the time, space and velocity step meshes. We consider the following discretization of the VFP equation
\begin{align}
f^{(1)}_{ij} &= f^n_{ij},\nonumber\\
f^{(2)}_{ij} &= f^n_{ij} + \frac{\Delta t}2 F^{(1)}_{ij},\nonumber\\
f^{n+1}_{ij}&=f^n_{ij} + \Delta t F^{(2)}_{ij}, \label{eq:fnp1}
\end{align}
where,
$$
F^{(k)}_{ij}=-v_jD_xf^{(k)}_{ij} - E^{(k)}_iD_vf^{(k)}_{ij} + \nu Q^{(k)}_{ij}, \text{ with } E^{(1)}_i=E^n_i.
$$
The time integrator used here is a 2-stage second order explicit RK method with $a_{21}=1$ and weights $b_1=0,\,b_2=1$ (explicit midpoint rule). The aim is to find an adapted approximation of $E^{(2)}_i$ that guarantees exact conservation of the total energy
$$
\mathcal{E}_{tot}(t) = \mathcal{E}_{kin}(t) + \mathcal{E}_{elec}(t),
$$
where 
$$
\mathcal{E}_{kin}(t)=\frac12\int_{\R\times[0,L]}|v|^2fdxdv \quad\text{ and }\quad \mathcal{E}_{elec}(t) = \frac12\int_0^{L}E^2(t,x)dx.
$$
In other words, we need to satisfy the equality $\frac{d\mathcal{E}_{tot}(t)}{dt}=0$ at the numerical level. If we multiply \eqref{eq:fnp1} by $\frac{v_j^2}2\Delta x\Delta v$ and we sum over indices $i\in{1,\dots,N_x},\,j\in\Z$, we get
\begin{align*}
\frac{\kin^{n+1}-\kin^n}{\Delta t} = \left(-\frac12\sum_{i,j}v_j^3D_xf^{(2)}_{ij}-\frac12\sum_{i,j}E^{(2)}_iv_j^2D_vf^{(2)}_{ij}+\frac{\nu}2\sum_{i,j}v_j^2Q^{(k)}_{ij}\right)\Delta x\Delta v,
\end{align*}
the first sum vanishes thanks to periodic boundary conditions in space, and the last sum also vanishes by Proposition \ref{prop_conservation}. Now, using discrete integration by parts we get
\begin{align*}
\frac{\kin^{n+1}-\kin^n}{\Delta t} &= \Delta x\Delta v\sum_{i,j}E^{(2)}_iv_jf^{(2)}_{ij}\\
&=-\Delta x\sum_{i}E^{(2)}_i\frac{E^{n+1}_i-E^n_i}{\Delta t},
\end{align*}
where $E^{n+1}_i=E^n_i-\Delta t\Delta v\sum_{j}v_jf^{(2)}_{ij}$ is an approximation of the electric field at time $t^{n+1}$ using the midpoint rule quadrature applied to Amp\`ere equation. Thus, we consider the following approximation of the electric field
\begin{equation}\label{eq:exact}
E^{(2)}_i=\frac{E^{n+1}_i+E^n_i}{2},
\end{equation}
which implies the conservation of the total energy
\begin{equation*}
\frac{\kin^{n+1}-\kin^n}{\Delta t} +{\frac{ \Delta x}{2}}\sum_{i}\frac{(E_i^{n+1})^2-(E_i^n)^2}{\Delta t} = 0.
\end{equation*}

{The scheme \eqref{eq:fnp1}  uses RK2 for the whole equation. The scheme that we would like to consider is the one described in Section \ref{sec:srr} (SL -- RK2 -- RKC2). }

\begin{remark}
	A high order upwind scheme is used to approximate the velocity derivative instead of a centered scheme as in 
	\cite{cf} which improves the stability of the resulting scheme, but the strategy can be extended to the Vlasov-Maxwell case. 
\end{remark}

\section{Numerical results}
This section is dedicated to the presentation of some numerical tests in different configurations: homogeneous in space, one-dimension   in space and velocity and one-dimension in space and two dimensions in velocity. Note that if we write RKC without precising the order,  then we mean the second order method.

\subsection{Homogeneous case }
First, we consider the simple homogeneous problem in one dimension satisfied by $f(t, v)$ 
\begin{equation}
\partial_t f = \nu \partial_v ((v-u) f + T\partial_v f),
\label{eq:hom} 
\end{equation}
with the initial condition 
\begin{equation} 
\label{init_hom}
f(t=0, v) =  M_{\alpha, 0, T_c}(v)+ M_{\frac{1-\alpha}{2}, 4, 1}(v) + M_{\frac{1-\alpha}{2}, -4, 1}(v), \;\; \mbox{ with }  \;\; M_{\rho, u, T} = \frac{\rho}{\sqrt{2\pi T}}e^{-\frac{(v-u)^2}{2T}}. 
\end{equation}  
with $T_c=0.2$ and for the current section, we choose the following initial data with $\alpha=0.9$. 
We consider a velocity domain $[-v_{\max}, v_{\max}]$ with $v_j=-v_{\max}+j\Delta v$, $\Delta v=2v_{\max}/N_v$, and $N_v\in\mathbb{N}$. 
The collision operator is discretized by the scheme \eqref{form1} with Dirichlet boundary condition at $v_{\max}$.  

\subsubsection{CFL condition and second order in time}
Let us discuss in practice the advantages of RKC methods applied to $\partial_t f = \nu Q(f)$, where $\nu=0.1$, $v_{\max}=12$, $N_v=256$, $s=20$ stages. The preserved quantities $T=1.8822$, $u=0$ are calculated using the initial condition. For RKC1, the time step must be chosen such that
$$
\Delta t\leq\frac{1.93s^2}{4\nu T}\Delta v^2\simeq 9.09,
$$
and for RKC2, we get
$$
\Delta t\leq\frac{0.65s^2}{4\nu T}\Delta v^2\simeq 3.06.
$$
The CFL condition for standard explicit methods such as explicit Euler or RK2 reads
$$
\Delta t\leq\frac{1}{2\nu T}\Delta v^2\simeq 0.023.
$$
This means that for a final time $t_{\max}=100$, the cost of RKC1 is: $(t_{\max}/\Delta t)\times s = (100/9.09)\times20\simeq221$ evaluations of $Q$, and for RKC2: $(100/3.06)\times20\simeq654$ evaluations of $Q$. On the other hand, the cost of explicit Euler method: $100/0.023\simeq4348$ evaluations of $Q$ and that of RK2:$(100/0.023)\times 2\simeq8696$ evaluations of $Q$. 
Note that here we are not taking into account the possible restriction coming from the advection in $v$, since for the above values the eigenvalues  are all on the negative real axis (as in Figure \ref{fig:ev}a). To illustrate this for the RKC2 method, we plot in Figure \ref{fig:cfl} the solution for three different time steps at final time $t_{\max}=300$. 
We observe that the solution reaches the Maxwellian equilibrium for $\Delta t=0.5, 3.05$ but for $\Delta t=3.085$, the method becomes unstable, as predicted by the 
previous analysis. 

\begin{figure}[tb]
	\centering
	\begin{subfigure}[t]{0.32\textwidth}
		\includegraphics[width=\linewidth]{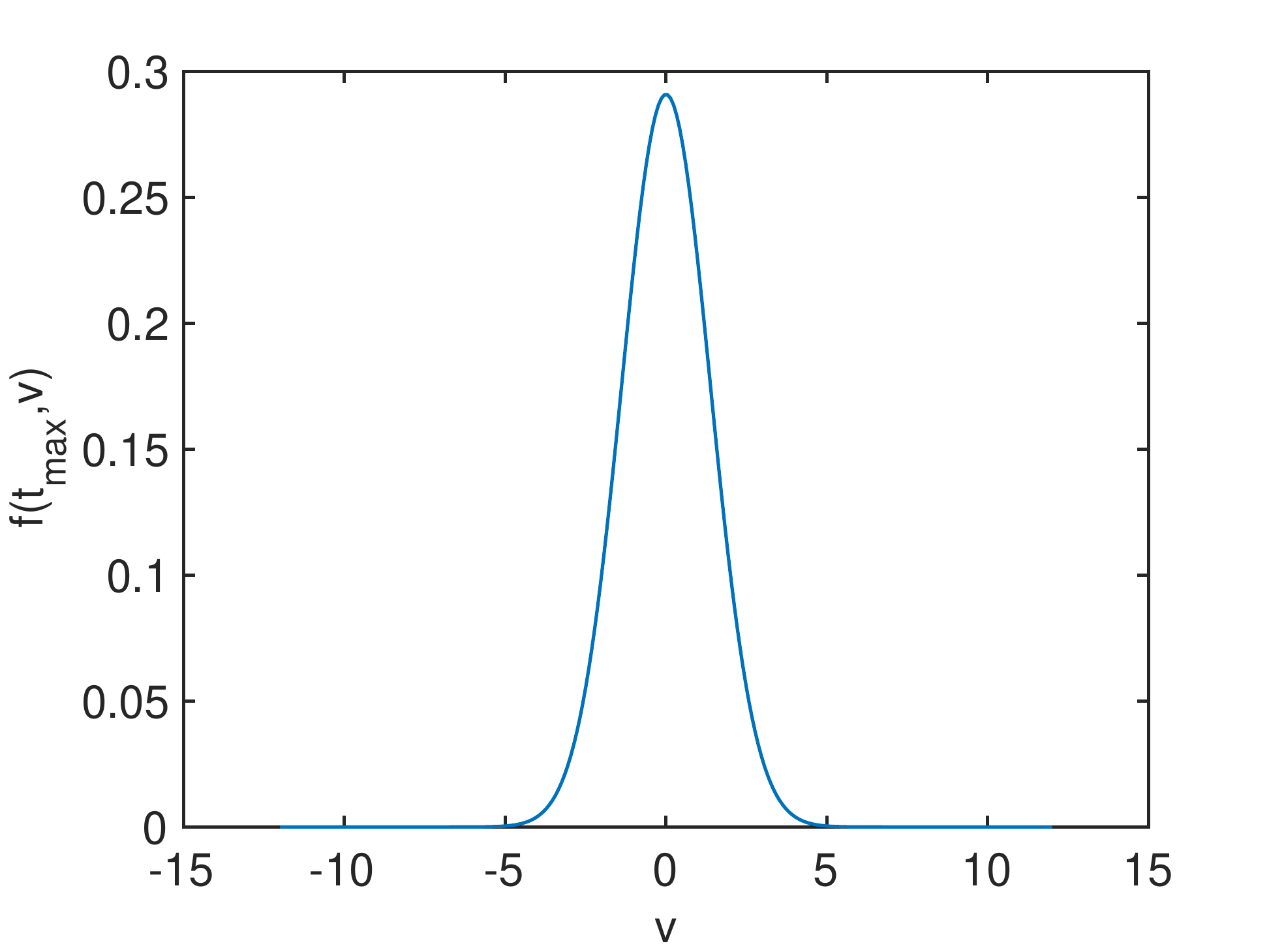}
		\caption{$\Delta t=0.5$.}
	\end{subfigure}
	\begin{subfigure}[t]{0.32\linewidth}
		\includegraphics[width=\linewidth]{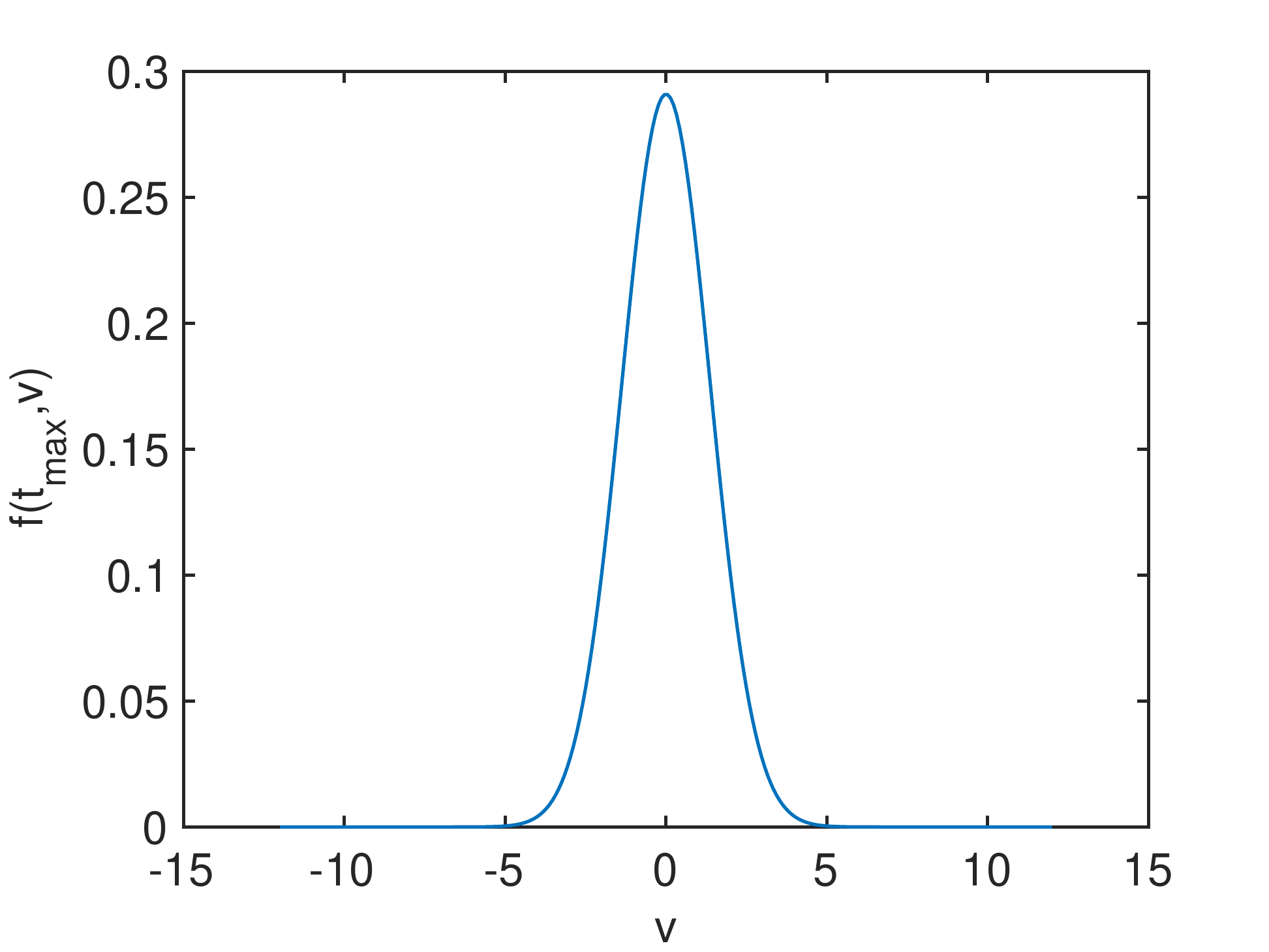}
		\caption{$\Delta t=3.05$.}
	\end{subfigure}
	\begin{subfigure}[t]{0.32\linewidth}
		\includegraphics[width=\linewidth]{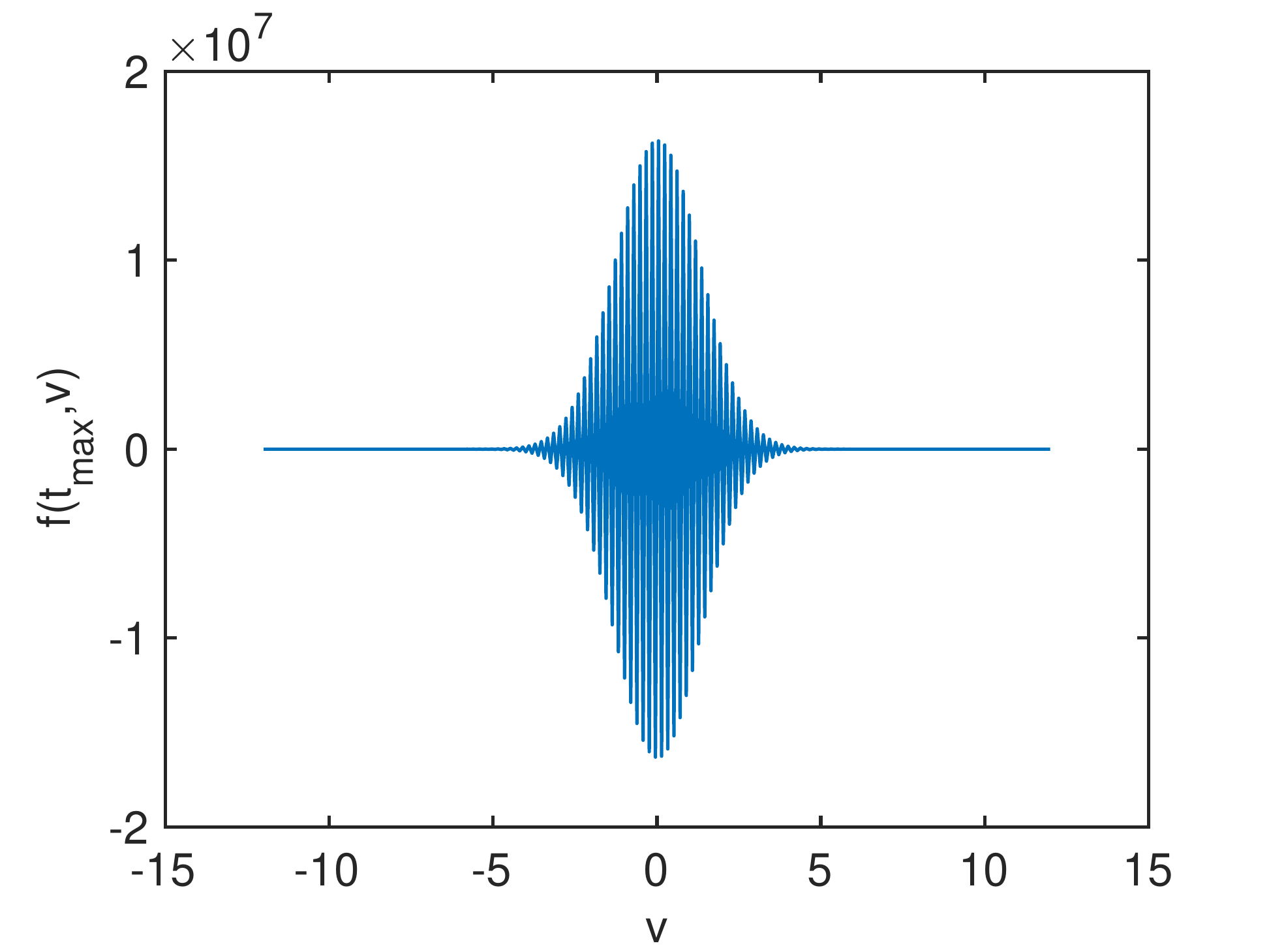}
		\caption{$\Delta t=3.085$.}
	\end{subfigure}
	\caption{Illustration of the CFL condition for RKC2: plots of $f(t_{\max}=300, v)$ with different time steps. }
	\label{fig:cfl}
\end{figure}

Now, still with $N_v=256$, $\nu=0.1$, $v_{\max}=12$, and final time $t_{max}=30$, we show in Figure \ref{fig:order2time} the second order in time of the RKC2 integrator. The values of $T$ and $u$ are calculated according to \eqref{eq:Tu}, and $s$ is calculated adaptively in order to guarantee stability according to strategies presented in \ref{adaptive}.  The used time steps are $\Delta t=0.75/2^j$, $j=1\dots,8$, and for the reference solution we use $\Delta t=3\times 10^{-4}$.

\begin{figure}[h!]
	\centering
	\includegraphics[scale=0.2]{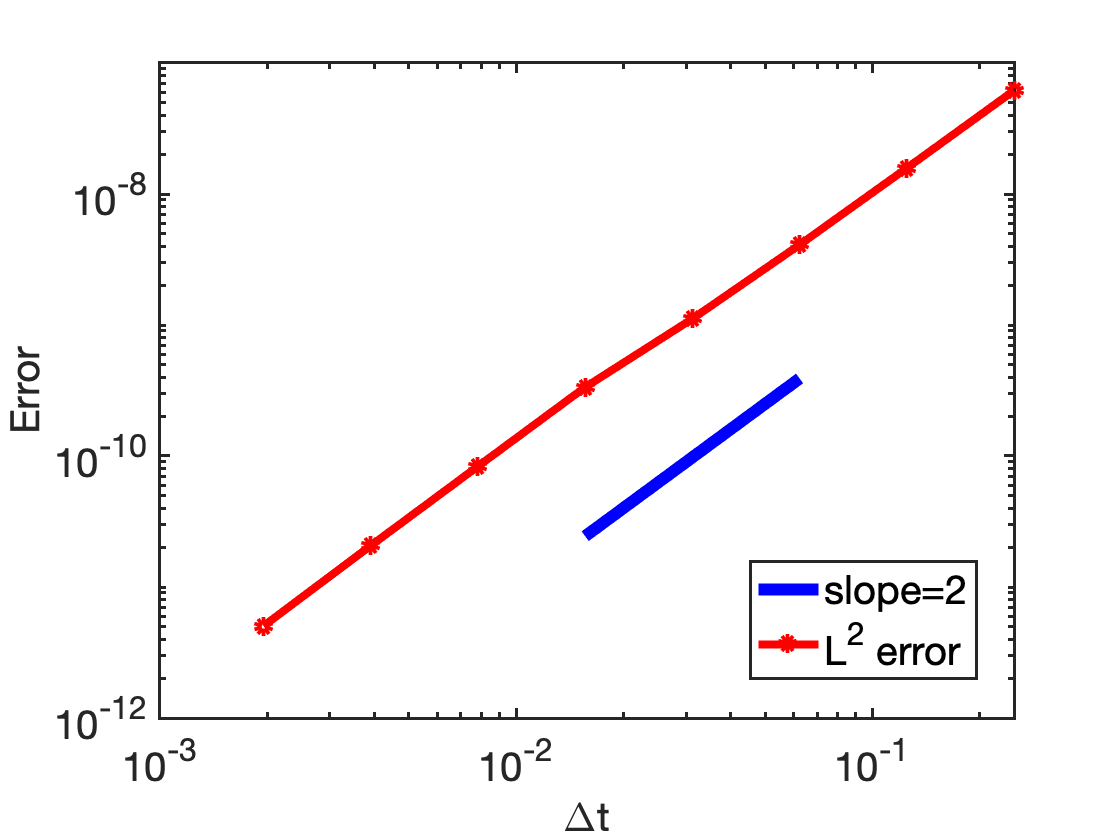}
	\caption{$L^2$ error at the final time $t_{\max}=30$ versus time step of the RKC2 scheme.}
	\label{fig:order2time}
\end{figure}

\subsubsection{Fourth order approximation}
We also consider the fourth order approximation \eqref{fourthorder} of the collision operator. 
First we illustrate the fourth order accuracy proven in Proposition \ref{fourth_order_prop}. In Figure \ref{fig:check_order4} (left),  
we plot the $L^2$ error 
$$
\Big(\sum_{j=0}^{N_v-1} |Q(f)(v_j)-Q_j|^2 \Delta v\Big)^{1/2}
$$ 
(for different $N_v=2^k, k=6, \dots, 13$ with $v_{\max}=14$) 
where $Q_j$ corresponds to the second or fourth order scheme, and $Q(f)(v_j)$ 
denotes the exact collision operator obtained for $f(v)=M_{0.1, 4, 1}(v)+M_{0.9, 0, 0.2}(v)$. The expected order of accuracy are recovered.  
Then, we consider the RKC method and we look at $f^\infty_j-M_{\rho, u, T}(v_j)$ where $(\rho, u, T)=(1, 0, 1.88)$ 
are the moments of the initial condition (which are preserved by the scheme) 
and $f^\infty_j$ is the numerical solution of \eqref{eq:hom} computed at a large time $t=1000$ so that the entropy does not change 
and the equilibrium is reached. We consider the following numerical parameters 
$\Delta t=2, s=20, v_{\max}=14$ and two different values for $N_v$.  The results in Figure \ref{fig:check_order4} (right)
show that the fourth order accuracy can be observed for the asymptotic numerical solution. Similar results are obtained 
using a standard RK2 time integrator, but using a much smaller time step.

\begin{figure}[h!]
	\centering
	$$
	\begin{array}{cc}
	\includegraphics[scale=0.2]{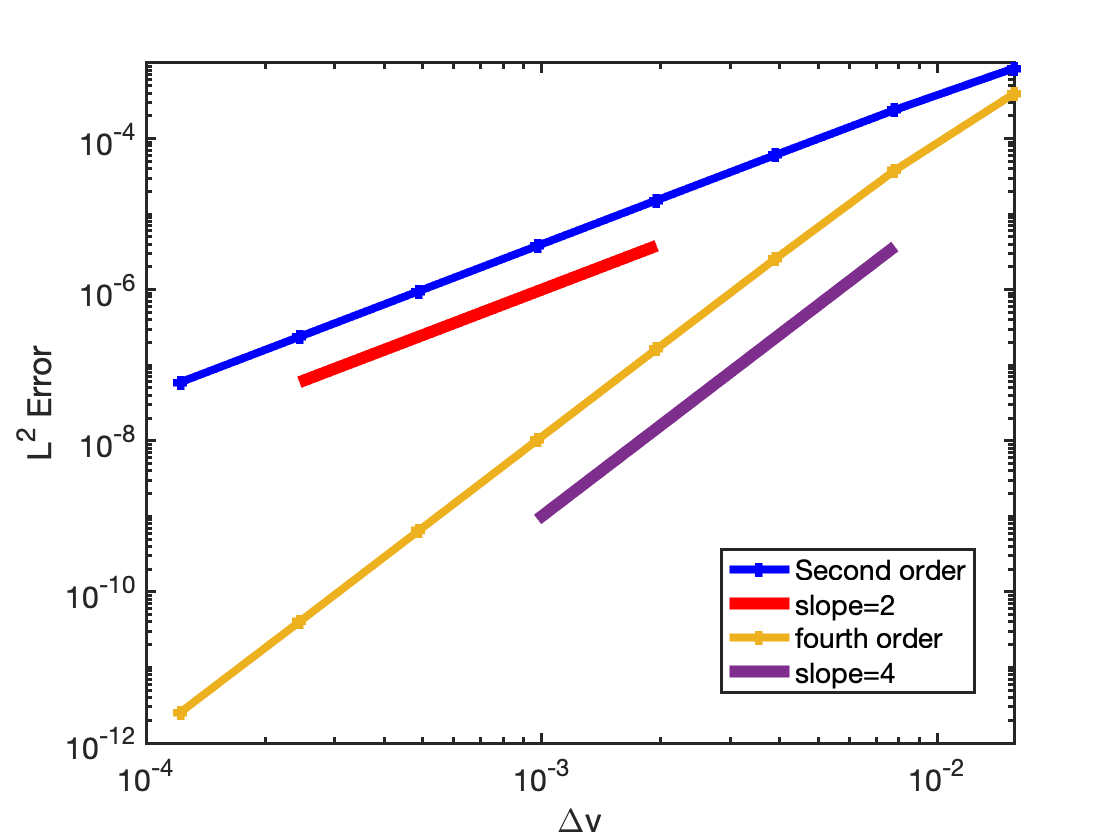} & \includegraphics[scale=0.2]{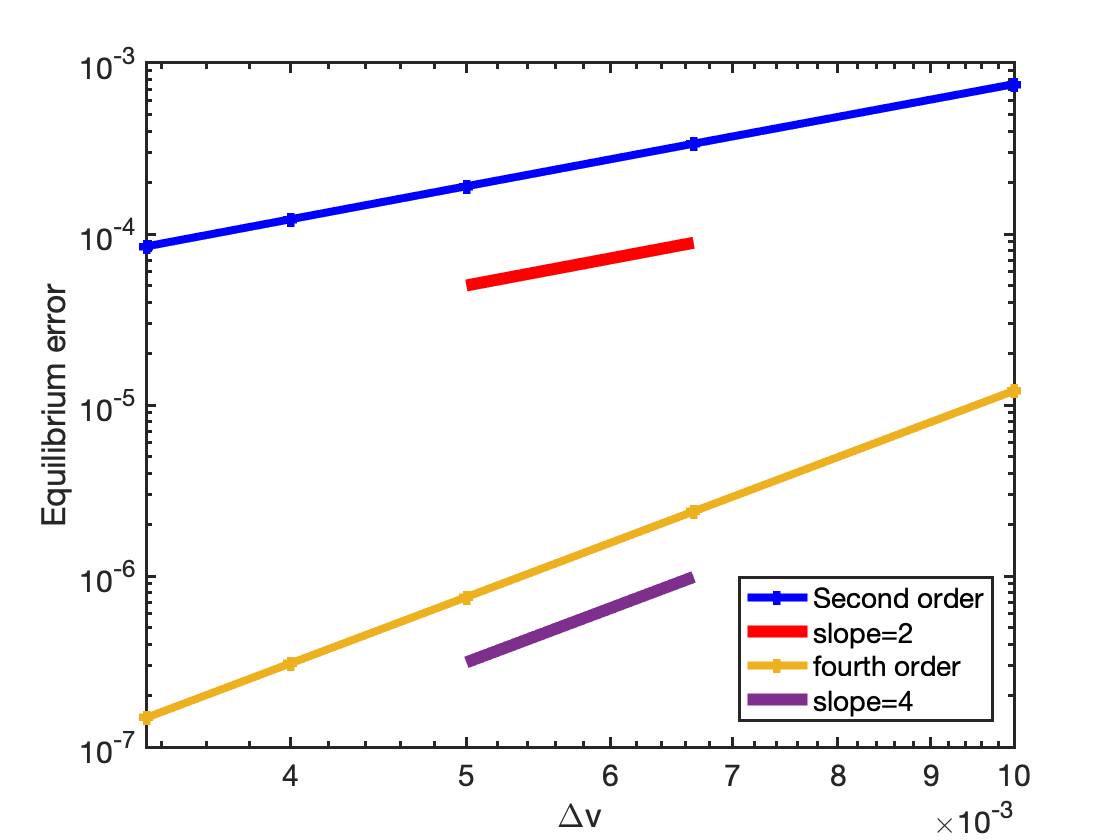} 
	\end{array}
	$$
	\caption{Left: $L^2$ error on $Q$ as a function of $N_v$ ($N_v=2^k, k=6, \dots,  13$) for the second order scheme \eqref{form1} and fourth order scheme \eqref{fourthorder}. 
		Right: $\max_j |f^\infty_j-M_{\rho, u, T}(v_j)|$ as a function of $N_v$ ($N_v=100, 150, 200, 250, 300$) for $\Delta t=2, t_f=1000$  for the second order scheme \eqref{form1} and fourth order scheme \eqref{fourthorder}.}
	\label{fig:check_order4}
\end{figure}


\subsubsection{Adaptive time stepping and conservation}
In this section, we test the performance of the algorithm with time step control strategy presented in \ref{adaptive} 
by considering the homogeneous problem \eqref{eq:hom}. The tolerance used throughout this section is $tol=10^{-6}$.

\paragraph{RKC2 versus RK2 }
We consider the same parameters as above ($N_v=256$, $\nu=0.1$, $v_{\max}=12$) and we compare the performance of the second order Runge's method (also called explicit midpoint rule or RK2) with Euler method for the error estimation and the RKC2 scheme with the error estimator explained in Section \ref{sec:est}. Table \ref{table:costexp} shows the advantage of RKC2 in terms of function evaluations and CPU time  for different values of final time. In Figure \ref{fig:history}, we plot the history of time steps for RKC2 and RK2 (semi-log scale) as well as the history of the number of stages used by RKC2 at each time step. The RK2 integrator faces a time step size restriction (CFL condition) which allows a maximum step size of around $10^{-2}$, while for RKC2 there is no such restriction which explains the advantage of stabilized methods illustrated in Table \ref{table:costexp}. Note that the number of rejected steps is negligible in all the cases for both methods.
\begin{table}
	\centering
	\begin{tabular}{ |c|c|c| } 
		\hline
		& RK2 & RKC2 \\ 
		\hline
		$Q$ evals & $8794$ & $1013$ \\ 
		\hline
		CPU time & $\approx 0.25s$ & $\approx 0.05s$ \\ 
		\hline
	\end{tabular}~~
	\begin{tabular}{ |c|c|c| } 
		\hline
		& RK2 & RKC2 \\ 
		\hline
		$Q$ evals & $42752$ & $1488$ \\ 
		\hline
		CPU time & $\approx 1.6s$ & $\approx 0.05s$ \\ 
		\hline
	\end{tabular}~~
	\begin{tabular}{ |c|c|c| } 
		\hline
		& RK2 & RKC2 \\ 
		\hline
		$Q$ evals & $85200$ & $1851$ \\ 
		\hline
		CPU time & $\approx 5s$ & $\approx 0.06s$ \\ 
		\hline
	\end{tabular}
	\captionof{table}{Cost comparison of the RKC2 and the RK2 methods for $t_{\max}=100$ (left), $t_{\max}=500$ (middle) and $t_{\max}=1000$ (right).}
	\label{table:costexp}
\end{table}

\begin{figure}[h!]
	\centering
	\begin{subfigure}{.49\textwidth}
		\includegraphics[width=\linewidth]{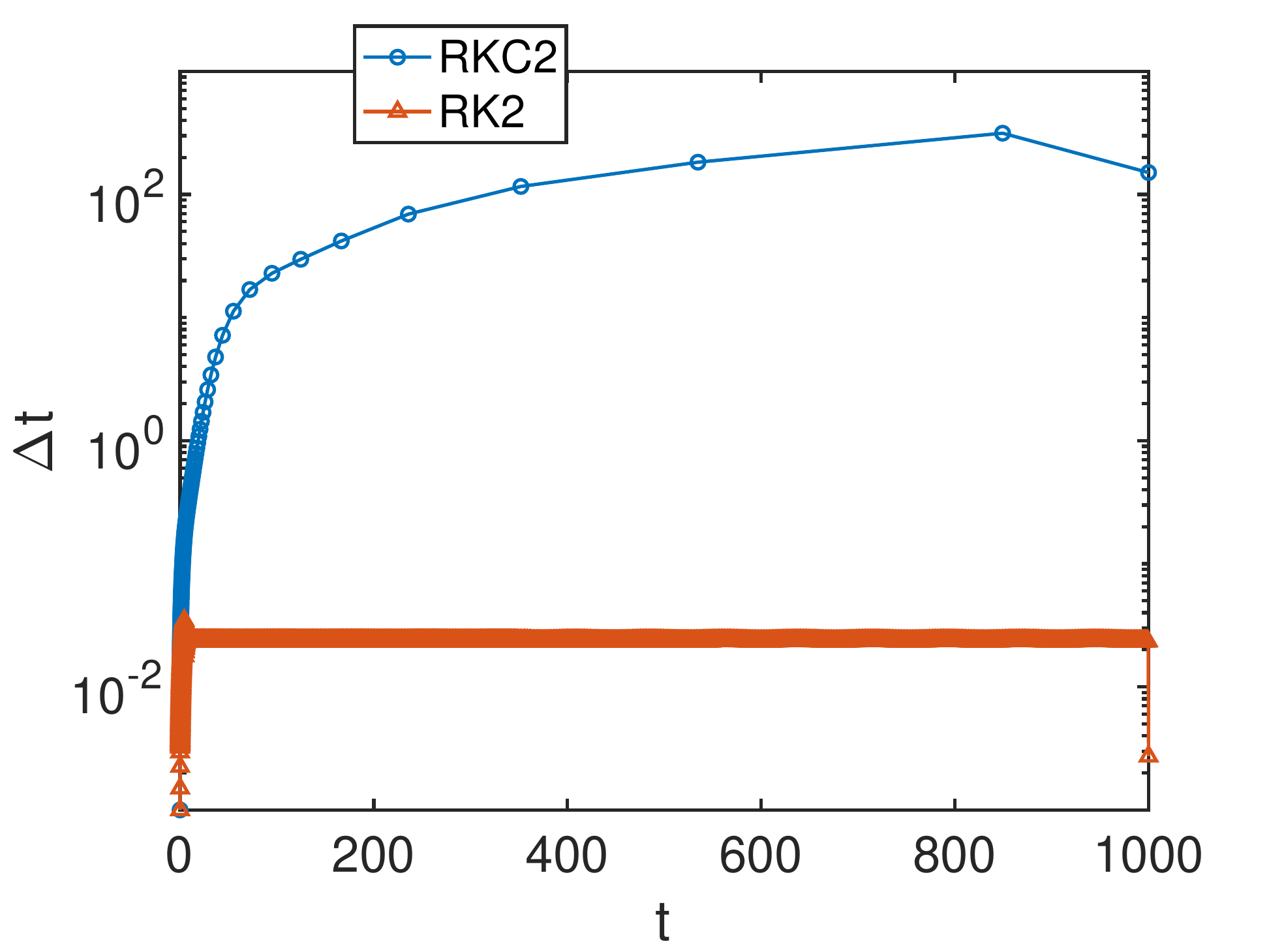}
		\caption{History of the time step size for RKC2 and RK2.}
	\end{subfigure}
	\begin{subfigure}{.49\textwidth}
		\includegraphics[width=\linewidth]{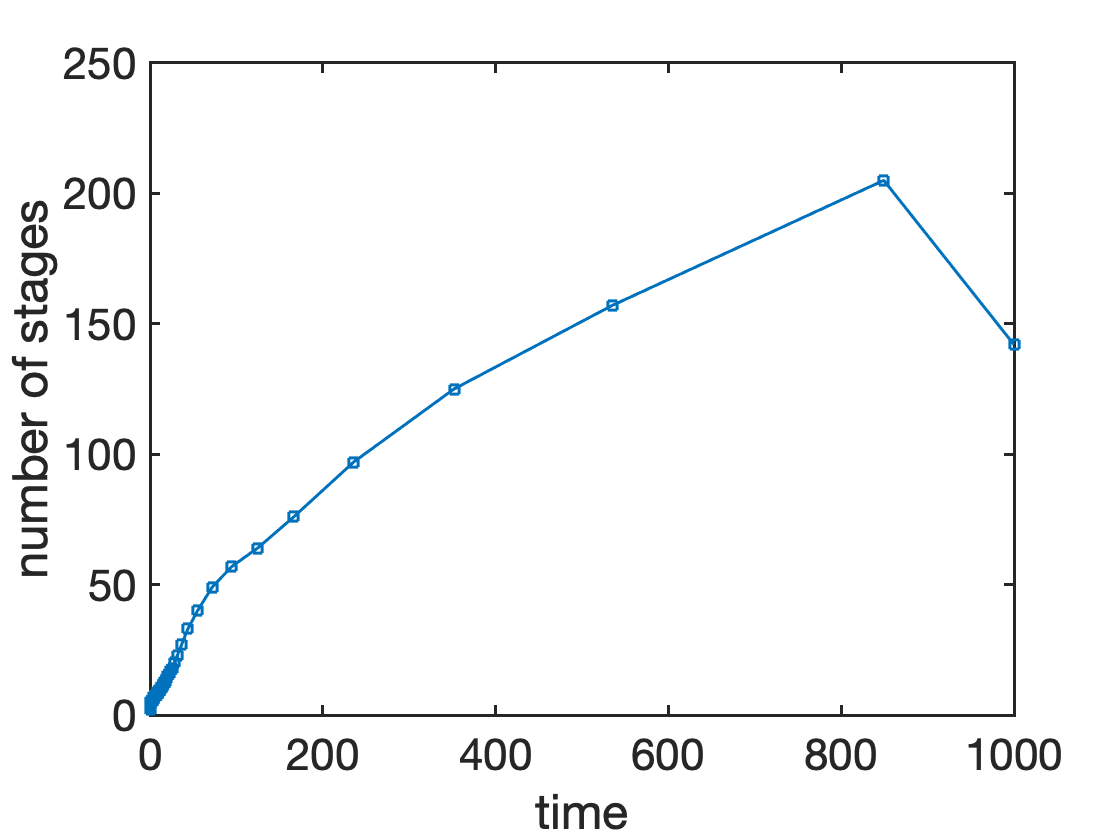}
		\caption{History of the stage number for RKC2.}
	\end{subfigure}
	\caption{History of the time steps and number of stages.}
	\label{fig:history}
\end{figure}

\paragraph{Conservation}
Finally, we investigate the properties of the velocity discretization coupled with the time integrator. To do so, we consider RKC2 with the adaptive time stepping 
presented in \ref{adaptive} with the following parameters $N_v=256$, $\nu=0.1$, $v_{\max}=16$, $tol=10^{-6}$ and $\Delta t_0=10^{-3}$ (note that 
$v_{\max}$ is larger to neglect the effect of the truncation ; otherwise, specific boundary conditions can be employed as in \cite{dellacherie}). 
In Figure \ref{fig:cons_hom}, the time history of  the deviation\footnote[1]{We define the deviation D(t) of a time dependent quantity Q(t) by D(t)=Q(t)-Q(0).} of mass, momentum and energy is displayed, as well as the time history of the 
entropy ${\cal E}(t) = \int f^2(t)/M dv$ and of $\|f(t, \cdot) - M(\cdot)\|_{L^2}$ (with $M$ the constant Maxwellian associated to the initial condition \eqref{init_hom}). 
First, we observe the conservation of the three invariants up to machine accuracy. Second, the entropy enjoys a very rapid decay  
relaxation of the solution towards an equilibrium which corresponds to the Maxwellian $M$ which shares the three first moments of \eqref{init_hom}  
as observed in the right of Figure  \ref{fig:cons_hom}. Even if the entropy decay can be proved in the semi-discrete case, it is not an easy task to prove it in the fully discrete case. 
\begin{figure}[tb]
	\centering
	\begin{subfigure}[t]{0.32\textwidth}
		\includegraphics[width=\linewidth]{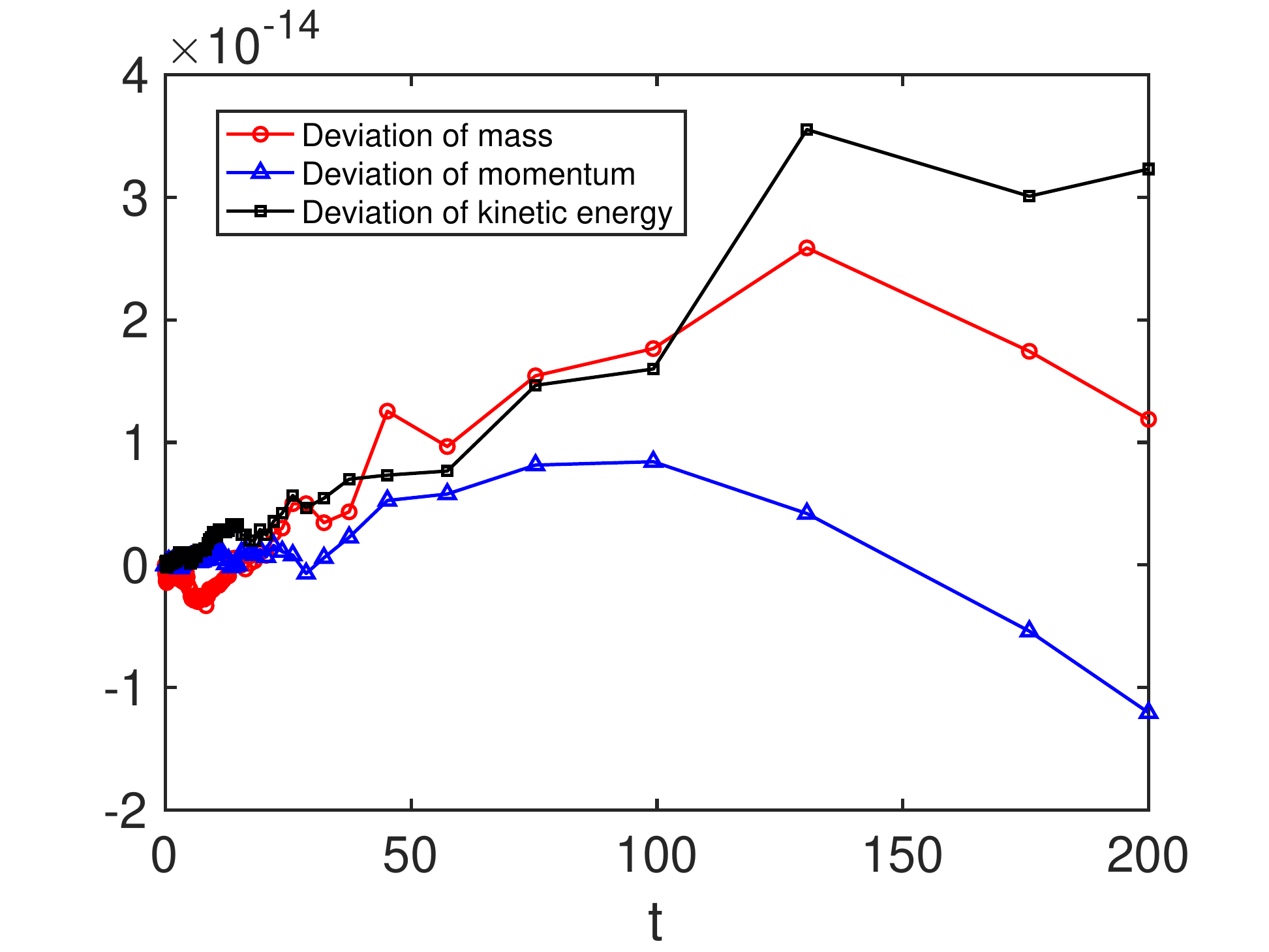} 		
	\end{subfigure}
	\begin{subfigure}[t]{0.32\linewidth}
		\includegraphics[width=\linewidth]{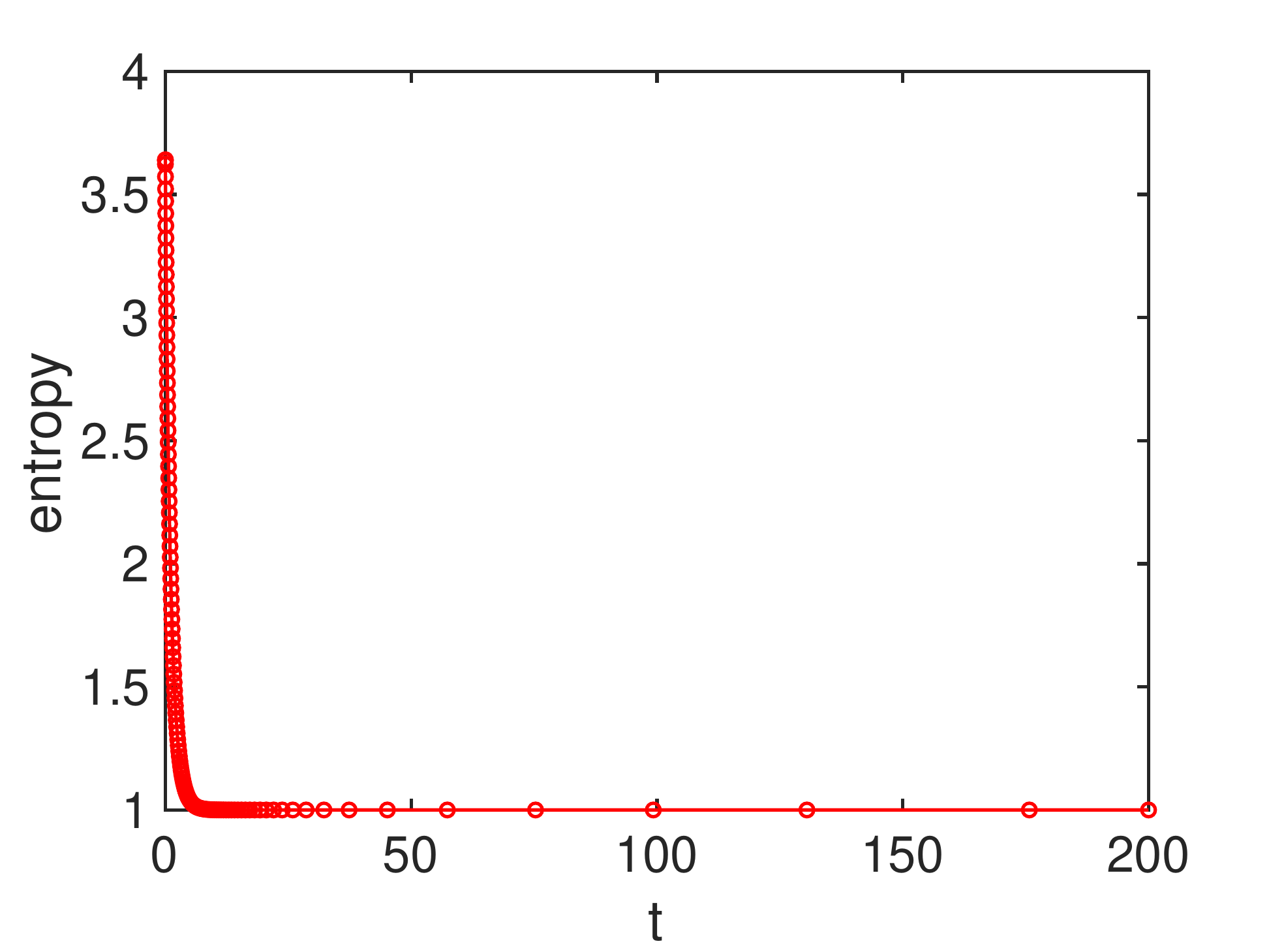}
	\end{subfigure}
	\begin{subfigure}[t]{0.32\linewidth}
		\includegraphics[width=\linewidth]{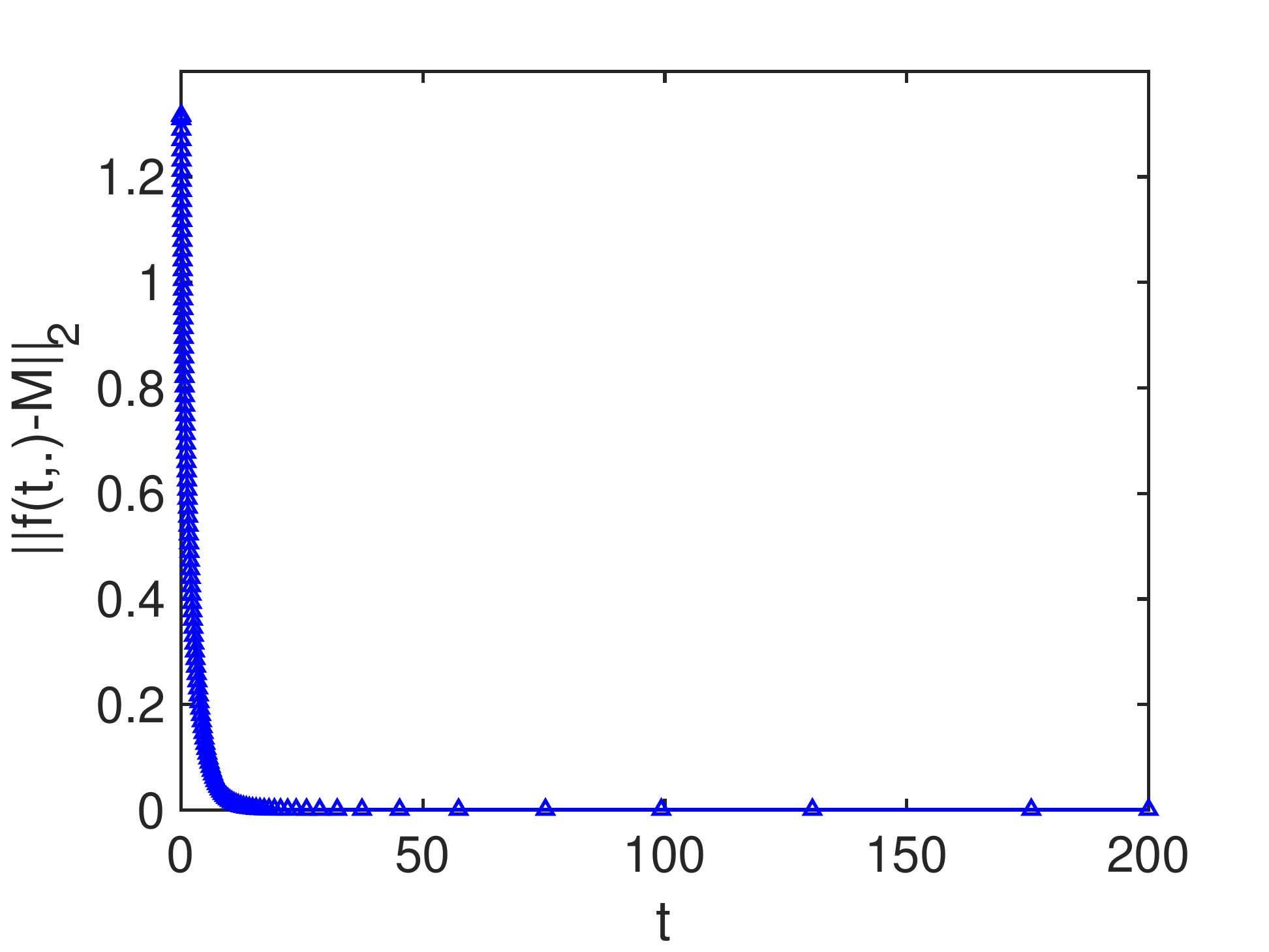}
	\end{subfigure}
	\caption{{Conservation properties of \eqref{eq:hom} solved with the RKC2 method.}  Left: Time history of the deviation of mass, momentum, and energy. Middle: Time evolution of ${\cal E}(t)=\int f^2/M dv$. Right: Time evolution of $\|f(t,\cdot) - M \|_{L^2}$.}
	\label{fig:cons_hom}
\end{figure}

\subsection{Non-homogeneous case 1dx-1dv}
Here, we solve the VPFP equation \eqref{vp_fp} for $d_x=d_v=1$ with the two different methods proposed in Section \ref{sec:coupling_vlasov}, namely 
SL-RKC and  SL-RK2-RKC. 
Different tests are proposed to illustrate the capability of the approach in the non-homogeneous case.

\paragraph{Test 1}
We start with the following initial data 
$$
f(0,x,v)=\frac{1}{\sqrt{2\pi }}e^{-\frac{v^2}{2}}(1+10^{-3}\cos(k x)), \;\; x\in [0, 2\pi/k], 
$$
with $k=0.5$. The results of this so-called (collisionless, $\nu=0$) Landau damping test are displayed on Figure \ref{fig:slope}: the time evolution of the electric energy (left) 
and the total mass, momentum and total energy. First, the expected damping rate is observed on the electric energy and the three invariants 
are preserved up to machine accuracy. 
The method used here is SL-RK2-RKC with $\Delta t=0.1$ is chosen quite small to ensure the stability of RK2 since in the collisionless regime the RKC is disabled and so its error estimator. 

\begin{figure}[h!]
	\centering
	\begin{subfigure}[t]{0.49\textwidth}
		\includegraphics[width=\linewidth]{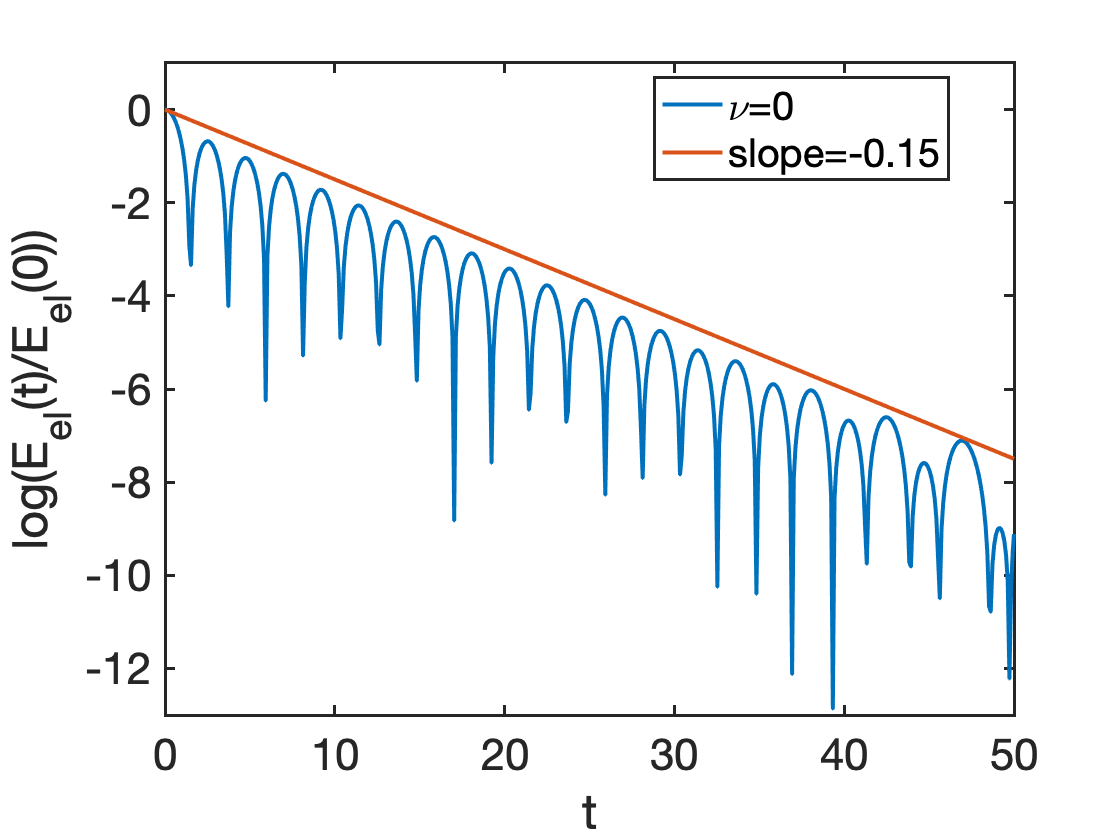}
		\caption{Time history of the electric energy.}
	\end{subfigure}
	\begin{subfigure}[t]{0.49\linewidth}
		\includegraphics[width=\linewidth]{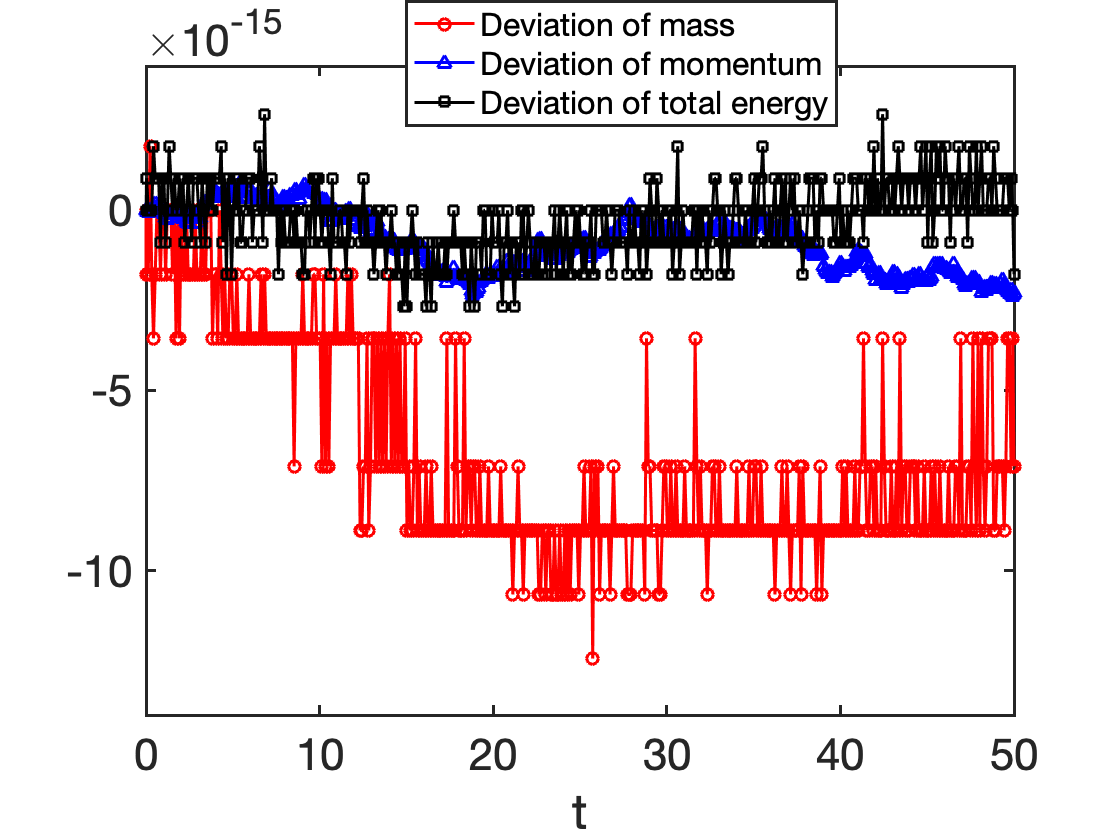}
		\caption{Time history of the deviation of total mass, momentum and total energy.}
	\end{subfigure}
	\caption{Test 1 (1dx-1dv) {solved with SL-RK2-RKC:} $N_x=128,~N_v=256,~v_{\max}=12,~\nu=0$.}
	\label{fig:slope}
\end{figure}

\paragraph{Test 2} 
Then, we consider 
$$
f_h(v)=\frac{1}{\rho_h\sqrt{2\pi}}v^\gamma e^{-\frac{v^2}{2}} \quad\text{where}
\quad \rho_h=\int_{v_{\min}}^{v_{\max}}\frac{1}{\sqrt{2\pi}}v^\gamma e^{-\frac{v^2}{2}}dv, \;\;\; f_c=\frac{1}{\sqrt{2\pi T_c}}e^{-\frac{v^2}{2T_c}}, 
$$
so that the initial data is 
$$
f(0,x,v)=(\alpha f_c(v) + (1-\alpha) f_h(v))(1+\beta\cos(k x)), 
$$
where $\gamma=10, \beta=k=0.5$, $T_c=0.2$, and $\alpha=0.9$. We consider here $\nu=0.1$, $v_{max}=14$, and $t_{max}=200$. {We use the adaptive time stepping strategy with initial time step equals to $10^{-3}$ and tolerance $10^{-6}$.}  



We compare the two approaches SL-RKC and SL-RK2-RKC. Both methods preserve mass and momentum but SL-RKC does not preserve the total energy 
(deviation of magnitude $10^{-4}$) whereas SL-RK2-RKC does. The time history of the mass, momentum and total energy are plotted in Figure \ref{fig:etot_dev}. 

{Regarding the cost, SL-RKC needed 360 time steps to perform the integration with the maximum allowed step size equals to $12.1$, while SL-RK2-RKC required 550 time steps with the maximum step size equals to $1.427$. The SL-RK2-RKC method needs more steps since the choice of the step size is still restricted by the advection CFL because of the use of RK2 for the term $E\partial_vf$. This term is integrated in the collision operator when using SL-RKC and it is usually small compared to the diffusion term which does not affect the stability of the scheme. 
}


\begin{figure}[h!]
	$$
	\begin{array}{cc}
	\includegraphics[scale=0.4]{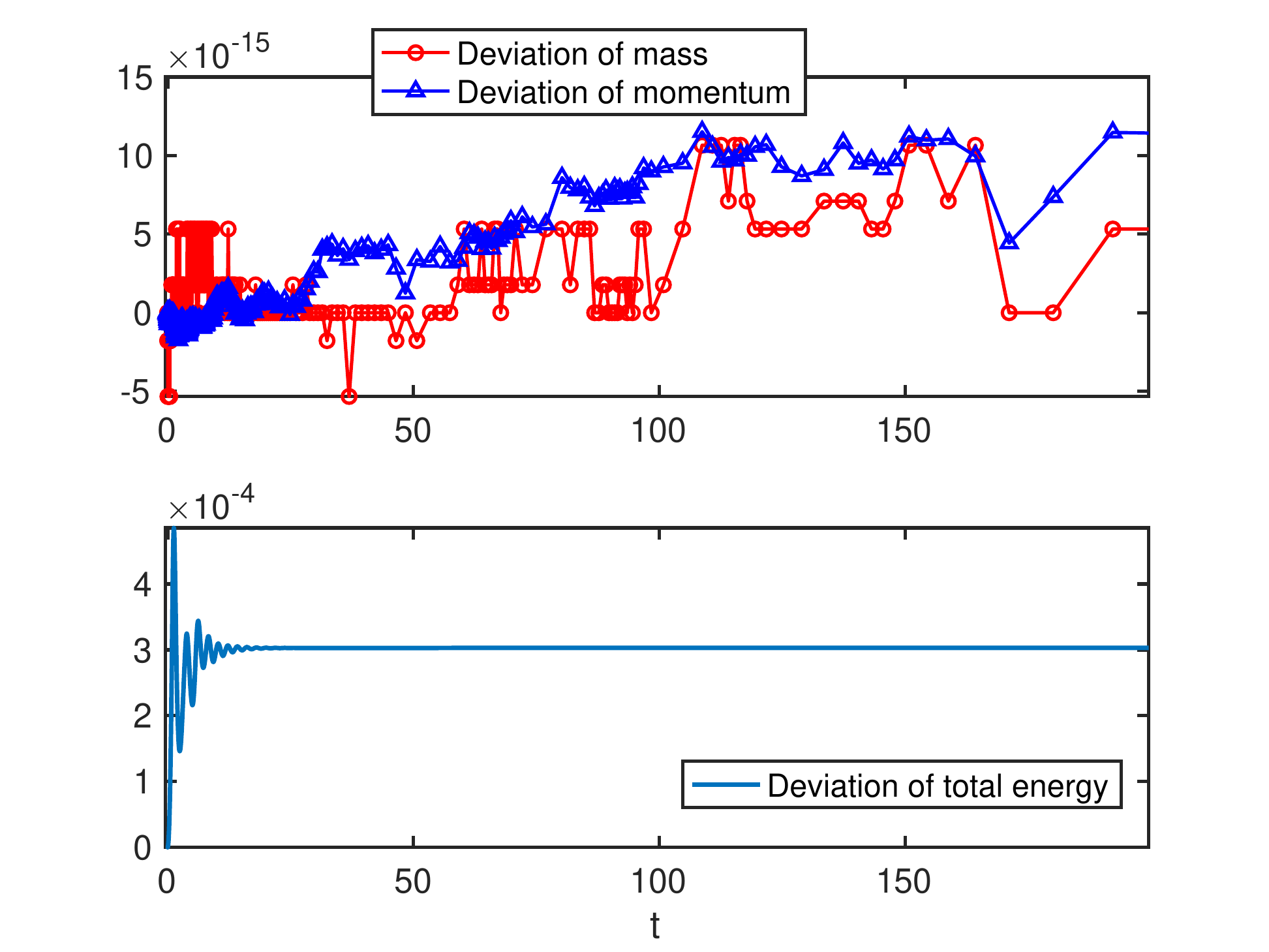}&
	\includegraphics[scale=0.4]{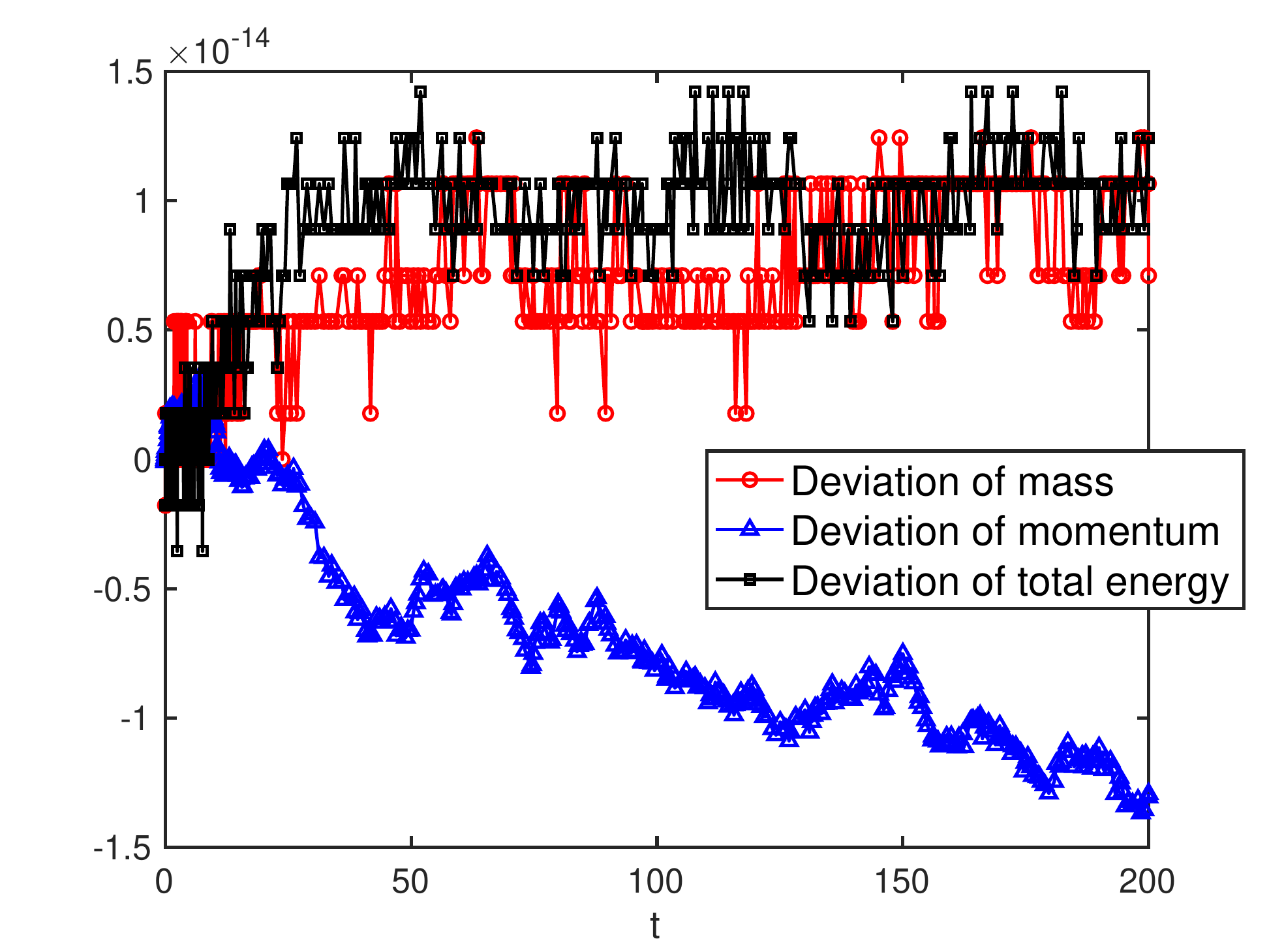}
	\end{array}
	$$
	\caption{Test 2 (1dx-1dv): Deviation of {mass, momentum}, and total energy for SL-RKC (left) and SL-RK2-RKC (right).}
	\label{fig:etot_dev}
\end{figure}


%

\paragraph{Test 3 }  
We consider a test from \cite{valentini} in which the initial condition is 
$$
f_0(x, v) = \Big(M_{n_1, 0, 1} + M_{n_2/2, u, 0.2} + M_{n_2/2, -u, 0.2}\Big) (1+\varepsilon \cos(kx)), 
$$
with $n_1=0.97, n_2=0.03, u=4, x\in [0, L], L=22, k=2\pi/L$, $\varepsilon=0.00056$. The numerical parameters 
are $N_x=128, N_v=256$ for the phase space discretization ($v_{\max}=14$) whereas $\Delta t=1$ thanks to the use of SL-RK2-RKC. 

In Figure \ref{fig:valentini_logelec}, we display the time evolution of the electric energy (semi-log scale) for a collisionless 
case (left) and with $\nu=5\times 10^{-4}$ (right). First, an exponential growth is observed with rate $\gamma=0.0746$  
which is very close to the theoretical value proposed in \cite{valentini}. Then, nonlinear effects saturate the exponential growth 
and the oscillations driving the particles trapping process are observed. When $\nu>0$, the evolution changes significantly: 
even if the exponential growth remains close to the collisionless case, the nonlinear phase is strongly affected by collisions 
since the saturation amplitude decreases as collisions smooth out the trapping effects and drives the system  towards 
a Maxwellian shape. The influence of the collisions on the trapping process can be seen in Figure \ref{fig:valentini_fxv} 
where we plot the phase space distribution function at time $t_1=80$ and $t_2=220$ in the collisionless and collision case. 
The time $t_1$ corresponds to the end of the linear phase and $t_2$ corresponds to the nonlinear phase. 
In the collisionless case (figure \ref{fig:valentini_fxv}, top row), the vortex around the phase velocity $v_\phi\approx 3.5$ 
can be observed whereas for $\nu>0$ (figure  \ref{fig:valentini_fxv}, bottom row), the trapping structure is smoothed out. 
Finally in Figure \ref{fig:valentini_etot}, the time evolution of total mass, momentum and energy is plotted 
and we can see that in both collisionless and collision cases, these three invariants are preserved up to machine accuracy. 

\begin{figure}[h!]
	$$
	\centering
	\begin{array}{cc}
	\hspace{-0.75cm}\includegraphics[scale=0.2]{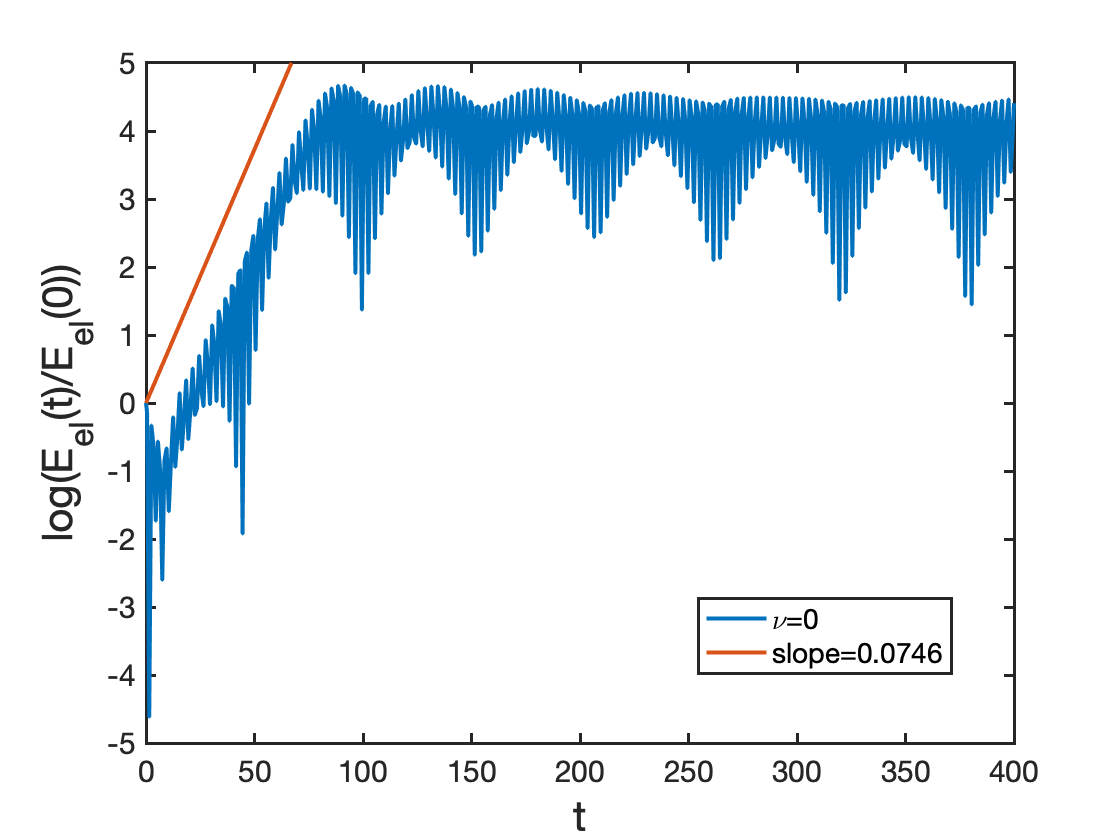} & \hspace{-0.85cm}\includegraphics[scale=0.2]{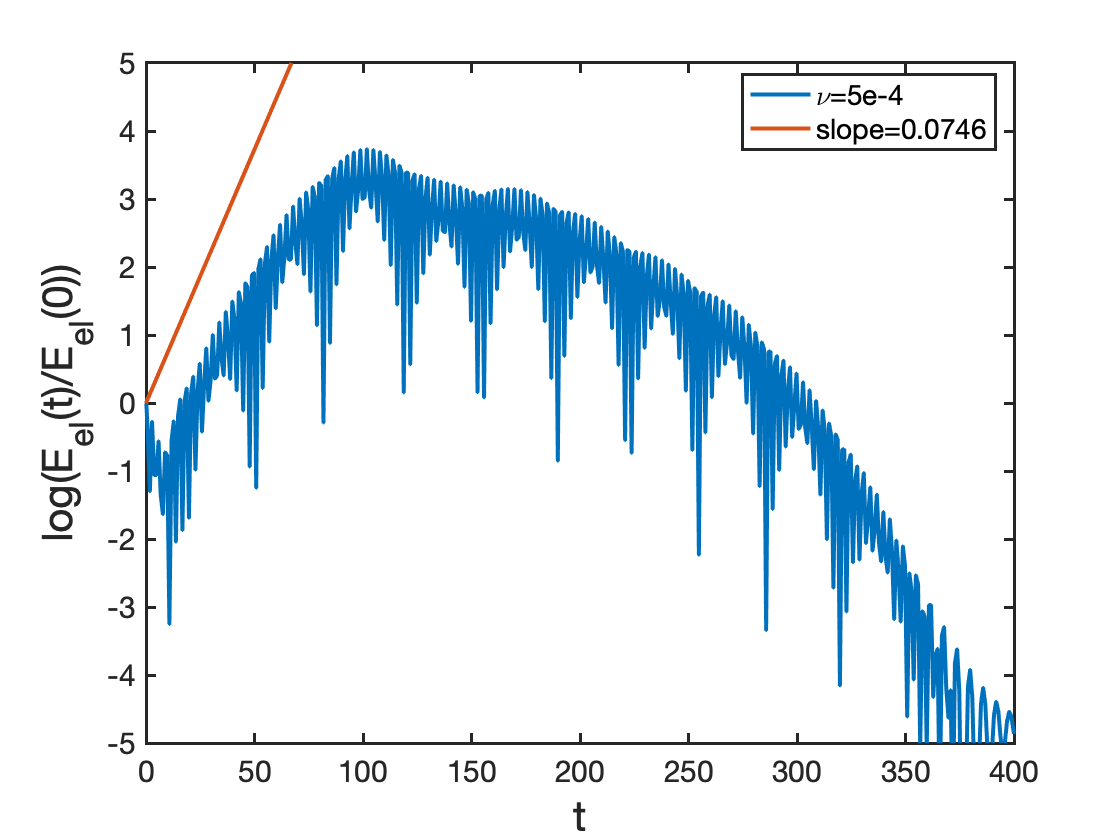}
	\end{array}
	$$
	\caption{Test 3 (1dx-1dv): $N_x=128$, $N_v=256$, $v_{\max}=14$, $\Delta t=1$. Left: $\nu=0$, right: $\nu=5\times 10^{-4}$ .}
	\label{fig:valentini_logelec}
\end{figure}

\begin{figure}[h!]
	$$
	\centering
	\begin{array}{ccll}
	\hspace{-0.75cm}\includegraphics[scale=0.2]{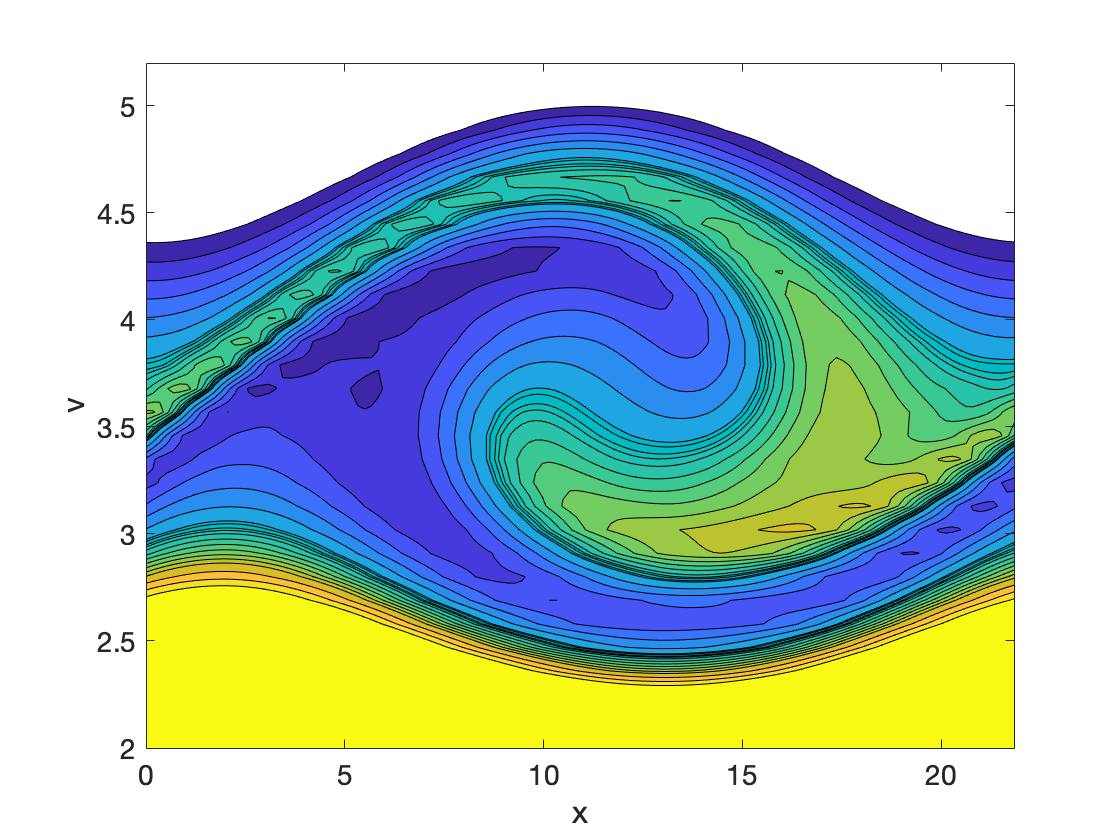} & \hspace{-0.85cm}\includegraphics[scale=0.2]{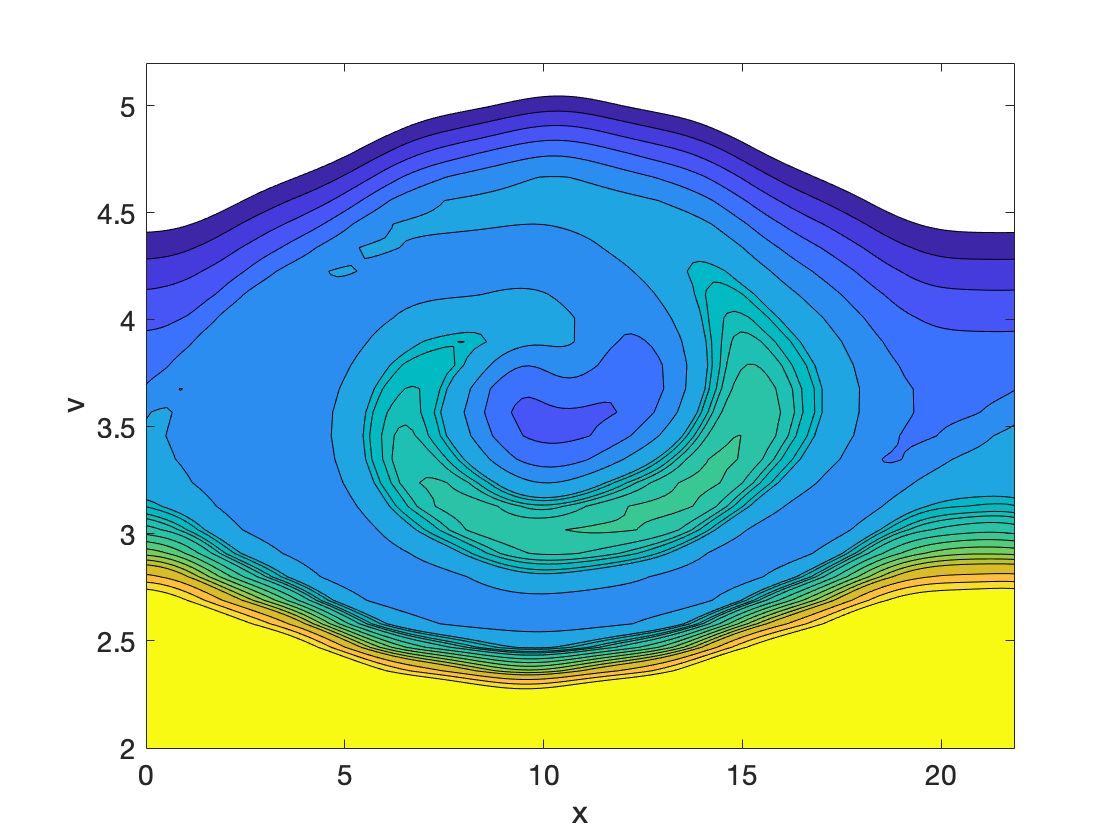}\\
	\hspace{-0.75cm}\includegraphics[scale=0.2]{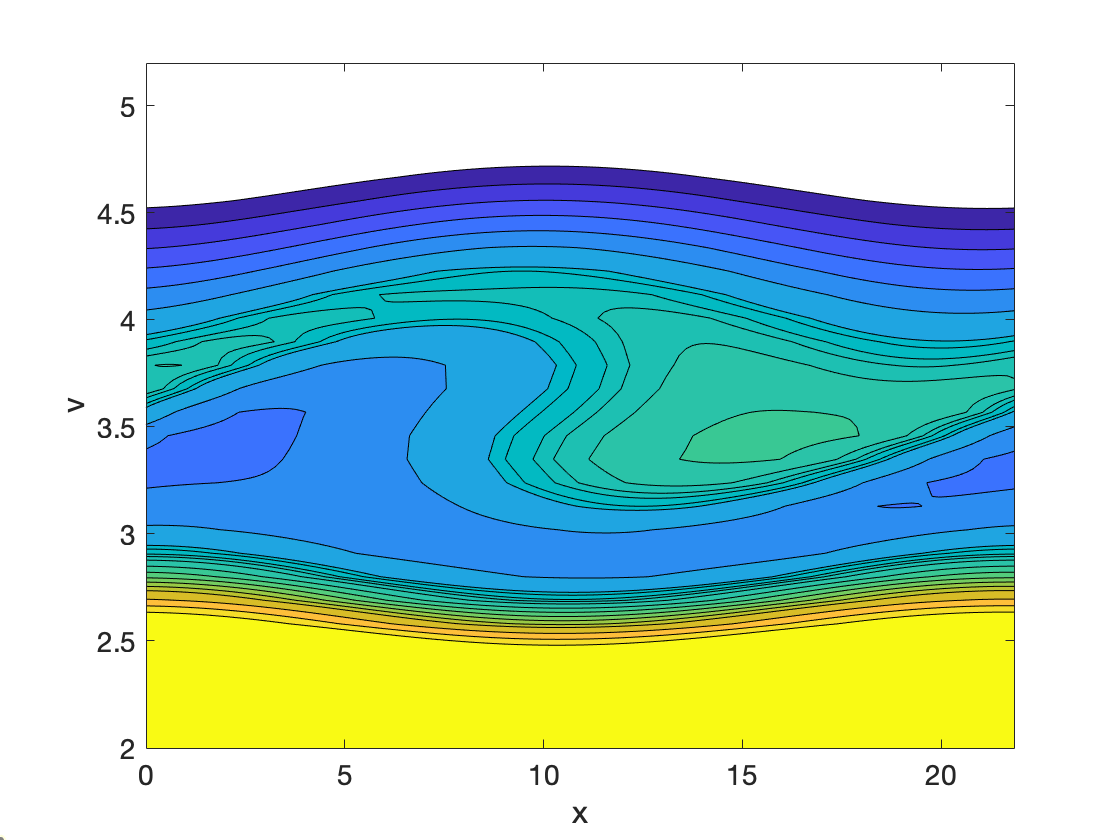} & \hspace{-0.85cm}\includegraphics[scale=0.2]{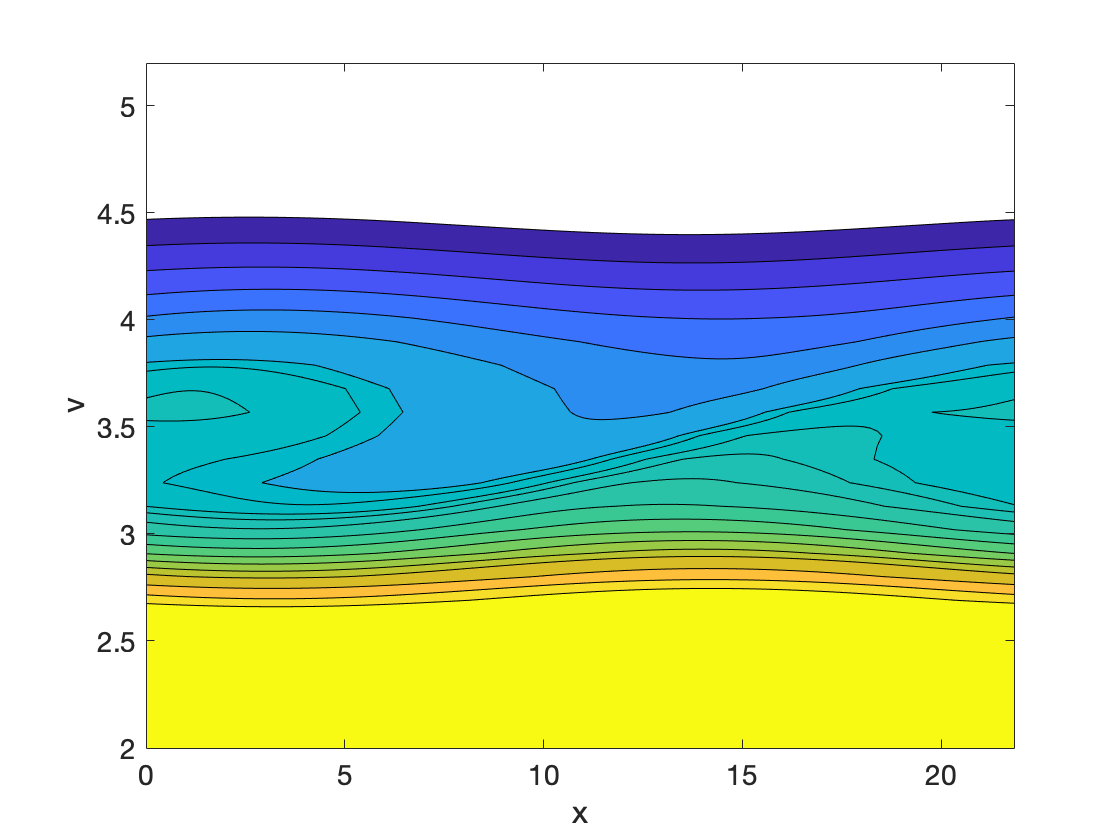}\
	\end{array}
	$$
	\caption{Test 3 (1dx-1dv): $N_x=128$, $N_v=256$, $v_{\max}=14$, $\Delta t=1$. First row: $\nu=0$, $t=80, 220$. Second row: $\nu=5\times 10^{-4}$, $t=80, 220$.}
	\label{fig:valentini_fxv}
\end{figure}


\begin{figure}[h!]
	$$
	\centering
	\begin{array}{cc}
	\hspace{-0.75cm}\includegraphics[scale=0.4]{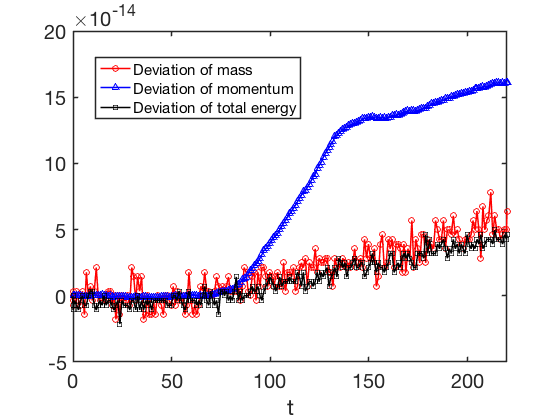} & \hspace{-0.85cm}\includegraphics[scale=0.4]{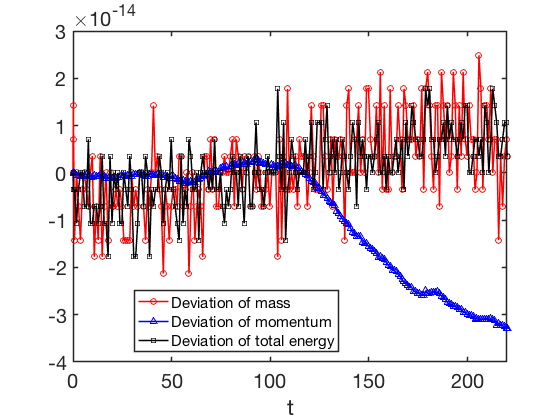}
	\end{array}
	$$
	\caption{Test 3 (1dx-1dv): $N_x=128$, $N_v=256$, $v_{\max}=14$, $\Delta t=1$. Time evolution of the total energy. Left: $\nu=0$, right: $\nu=5\times 10^{-4}$.}
	\label{fig:valentini_etot}
\end{figure}

\subsection{Non-homogeneous case 1dx-2dv}
Here, we solve the VPFP equation satisfied by $f(t, x, v)$ with $v=(v_x, v_y)$ 
$$
\partial_t f + v_x \partial_x f + E_x\partial_{v_x} f = \nu \nabla_v \cdot ((v-u_f)f + T_f\nabla_v f), 
$$
with $n_f u_f=\int vf dv_x dv_y$ and  $2 n_f  T_f=\int |v-u_f|^2 f dv_x dv_y$. 
In the following, we denote the Maxwellian by $M_{\rho, u_x, u_y, T}=\frac{\rho}{2\pi T}\exp(-[(v_x-u_x)^2+(v_y-u_y)^2]/(2T))$. 
The numerical scheme \eqref{form1} extends naturally to the multi-dimensional case and details are given in \ref{extension_2d}.
{In this part, we use three steps splitting (second order Strang version)  given by 
	\begin{itemize}
		\item Solve the transport in space $\partial_t f + v\partial_x f = 0$ using a semi-Lagrangian method or Fourier technique. 
		\item Solve the Poisson equation and solve the transport in velocity $\partial_t f + E\partial_v f = 0$ using a semi-Lagrangian method or Fourier technique. 
		\item Solve $\partial_t f = \nu Q(f)$ using several time integrators (RK2, implicit Euler or RKC). 
\end{itemize}}

\paragraph{Test 1}
The initial condition is given by 
$$
f(t=0, x, v_x, v_y) = M_{1, 0, 0, 1}  (1+10^{-4} \cos(k x)), 
$$
with  $\alpha=0.5$ whereas the spatial domain is $x\in [0, 2\pi/k]\,(k=0.3)$ and the velocity domain is $v=(v_x, v_y)\in [-7, 7]^2$. 
We investigate here the weakly collisional regime ($\nu$ small) on the Landau damping phenomena (see \cite{cf, tristani}). 

For the weakly collisional Landau damping, we consider $N_x=32, N_{v_x}=N_{v_y}=64$. The linear theory predicts a damping 
in the collisionless case, ie the root of the dispersion relation is $\omega=1.1598 - 0.0126i$. In Figure \ref{fig:ll_f2d}, we plot the time history of the electric energy (in semi-log sale) 
for RKC2 and $\Delta t=0.3, s=6$. Let us remark that with this time step, RK2 is unstable since it requires $\Delta t\approx 10^{-2}$. 
For different values of the collision frequency $\nu=0, 0.025, 0.05, 0.1$ we can observe that the electric energy is damped exponentially fast in time 
according to the collisionless Landau theory (the numerical damping $-0.012$ is in very good agreement with the linear theory $-0.0126$)  
in addition to the effect of the collisions which induces an additional damping rate (see  \cite{cf, tristani}).

\begin{figure}[h!]
	\centering
	\begin{subfigure}[t]{0.49\textwidth}
		\includegraphics[width=\linewidth]{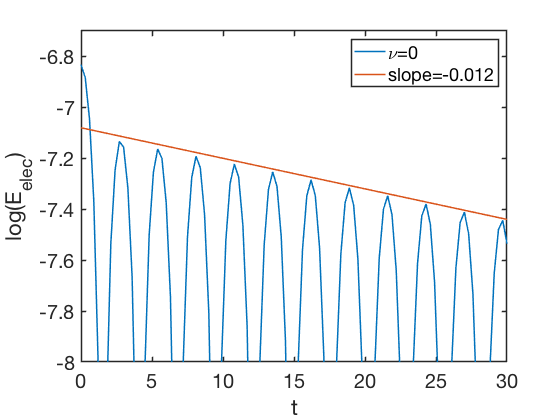}
	\end{subfigure}
	\begin{subfigure}[t]{0.49\linewidth}
		\includegraphics[width=\linewidth]{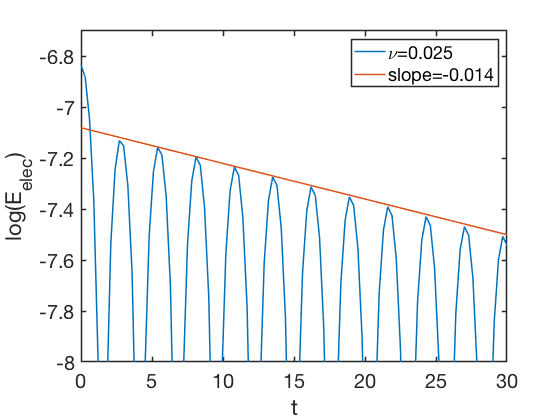}
	\end{subfigure}
	\begin{subfigure}[t]{0.49\linewidth}
		\includegraphics[width=\linewidth]{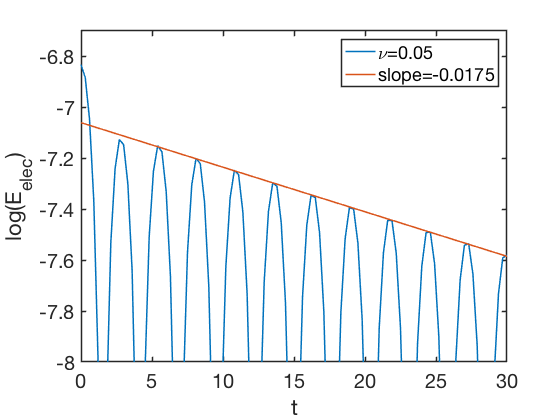}
	\end{subfigure}
	\begin{subfigure}[t]{0.49\linewidth}
		\includegraphics[width=\linewidth]{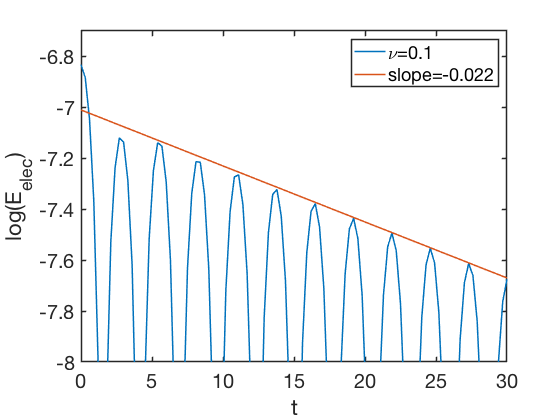}
	\end{subfigure}
	\caption{Test 1 (1dx-2dv): $N_x=32,~N_{v_x}=N_{v_y}=64,~v_{\max}=7, \Delta t=0.3, s=3$. Influence of $\nu=0, 0.025, 0.05, 0.1$ 
		on the time evolution of the electric energy.}
	\label{fig:ll_f2d}
\end{figure}

\paragraph{Test 2}
Lastly, we consider the following  initial condition 
$$
f(t=0, x, v_x, v_y) = \Big(\frac{(1-\alpha)}{4}\sum_{i=1}^4 M_{1, \pm 3, \pm 3, 1/2}  + \alpha M_{1, 0, 0, 1} \Big) (1+0.01 \cos(k x)), 
$$
with $\alpha=0.5$ whereas the spatial domain is $x\in [0, 2\pi/k]\, (k=0.5)$ and the velocity domain is $v=(v_x, v_y)\in [-18, 18]^2$. 
The velocity dependency of the initial condition is plotted in Figure \ref{fig:f2d} (left) for $x=\pi/k$.  We choose the collisions frequency $\nu=0.1$.

We compare different methods (applied on the collisional part) to emphasize the advantages of the stabilized Runge-Kutta methods. In Figure \ref{fig:5b_enelec}, we plot the time history of the electric energy in semi-log scale for different methods. A Strang splitting between the transport in space, in velocity, and the collision operator is used, where transports in space and in velocity are treated using semi-Lagrangian methods. 

First, in Figure \ref{fig:5b_enelec}a, we take $N_{v_x}=N_{v_y}=96$. For the the implicit method (implicit Euler only for the diffusion part of the collision operator) and the RKC methods we use $\Delta t=0.5$, while RK2 is run with $\Delta t=0.05$ for stability reasons. For RKC methods, we choose the smallest number of stages such that the scheme is stable ($s=3$ for RKC1 and $s=5$ for RKC2). For the explicit methods the decay stops before machine precision which means that the solution converges to a numerical equilibrium which is ${\cal O}(\Delta v^2)$ far from the exact equilibrium as expected.  Even if the implicit step might seem to be more optimized, the cost involved by this step turns out to be much more important compared to the explicit stabilized methods (factor $100$ in our (non optimized) implicit solver). On the other side, using a standard Runge-Kutta method (RK2) 
needs the use of small time steps to ensure stability. Finally, RKC1 or RKC2 can run with time steps comparable to the ones used by implicit methods and avoid solving large linear systems. 

In Figure \ref{fig:5b_enelec}b, we consider $N_{v_x}=N_{v_y}=128$ and we use $\Delta t=0.5$ for RKC methods ($s=4$ for RKC1 and $s=7$ for RKC2), while RK2 requires $\Delta t=0.025$. We can see that, for explicit methods,  the convergence to the Maxwellian state can be improved (to almost machine precision) by refining the velocity mesh. This shows again the high importance of stabilized methods since finer meshes increase the stiffness of the system which in turn increases the restriction on the time step for standard explicit methods, while stabilized methods can still run with the same large time steps but with just a few more internal stages. 

In Figure \ref{fig:f2d} (right), the velocity dependence of $f$ at time $t=50$ and $x=\pi/k$ is plotted which shows the relaxation of the non equilibrium initial condition 
to the Maxwellian state due to the effect of the collisions.


\begin{figure}[h!]
	\centering
	\includegraphics[width=\linewidth]{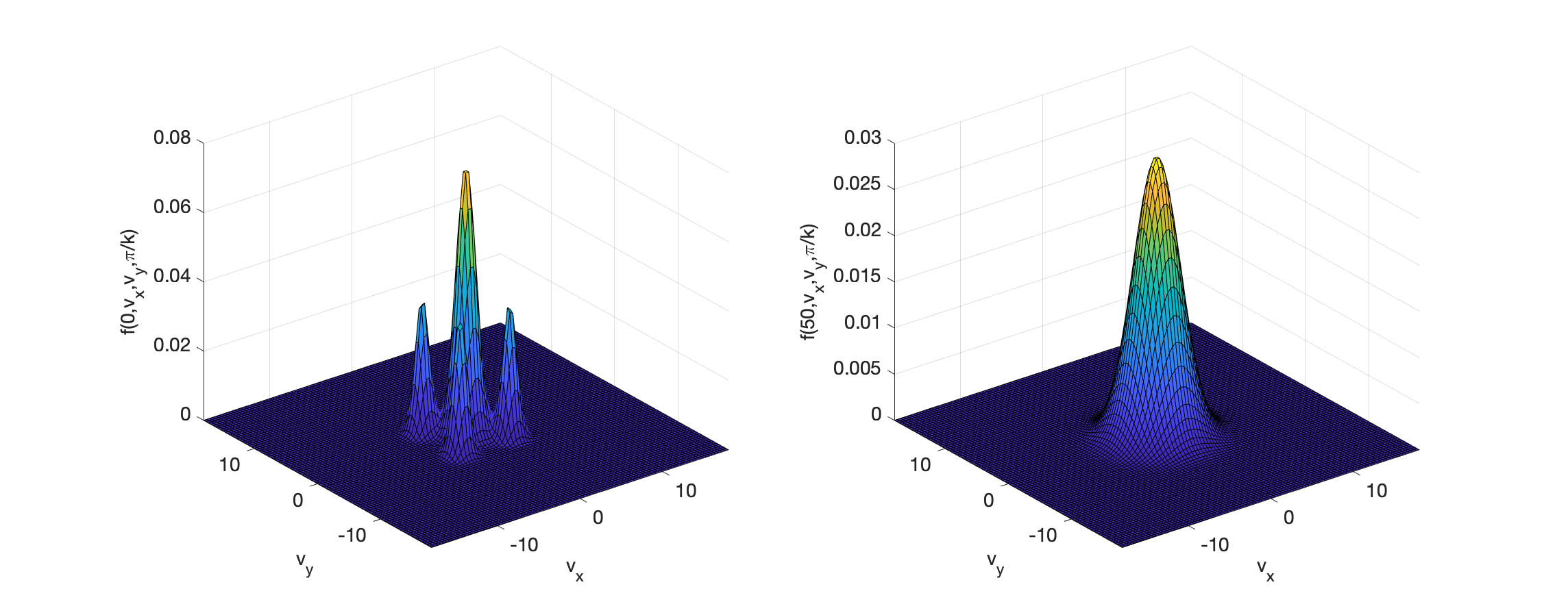}
	\caption{$N_x=32,~N_{v_x}=N_{v_y}=96,~v_{\max}=18$, $\nu=0.1$. Initial data $f(t=0, v_x, v_y, x=\pi/k)$ and final solution $f(t=50, v_x, v_y, x=\pi/k)$ .}
	\label{fig:f2d}
\end{figure}

\begin{figure}[h!]
	\centering
	\begin{subfigure}{0.49\linewidth}
		\includegraphics[width=\linewidth]{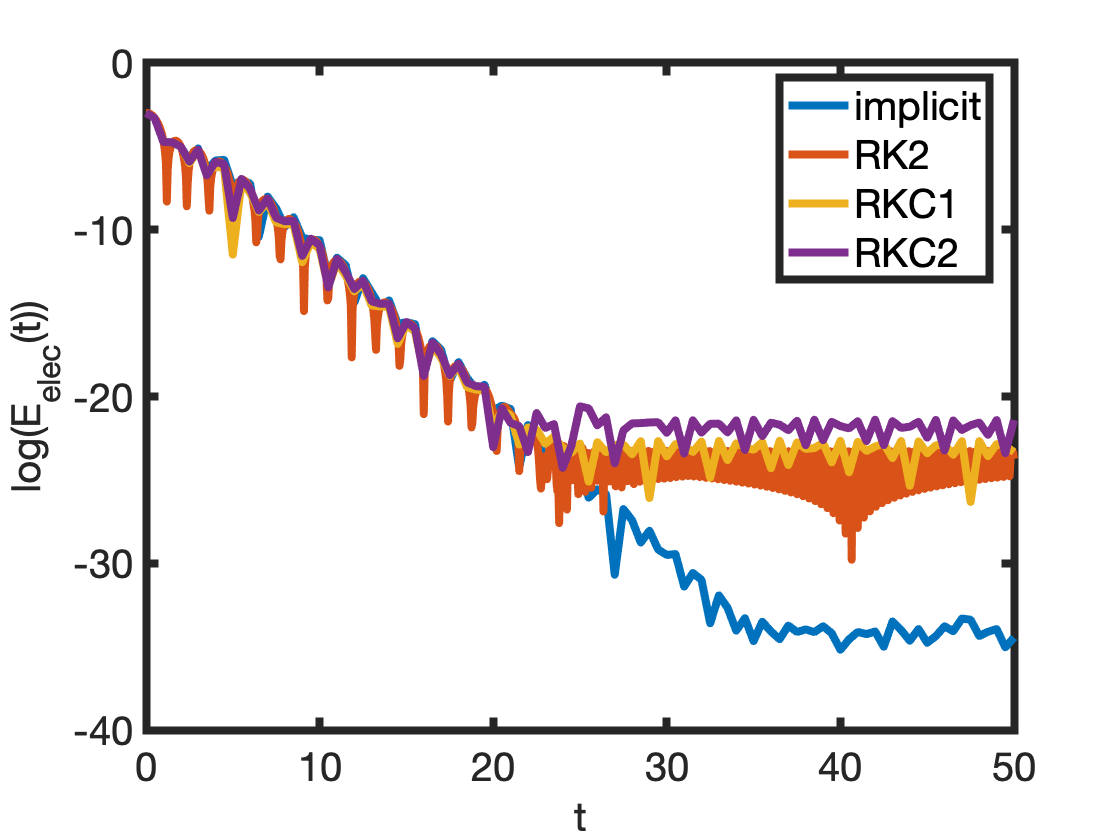}
		\caption{$N_{v_x}=N_{v_y}=96$.}
	\end{subfigure}
	\begin{subfigure}{0.49\linewidth}
		\includegraphics[width=\linewidth]{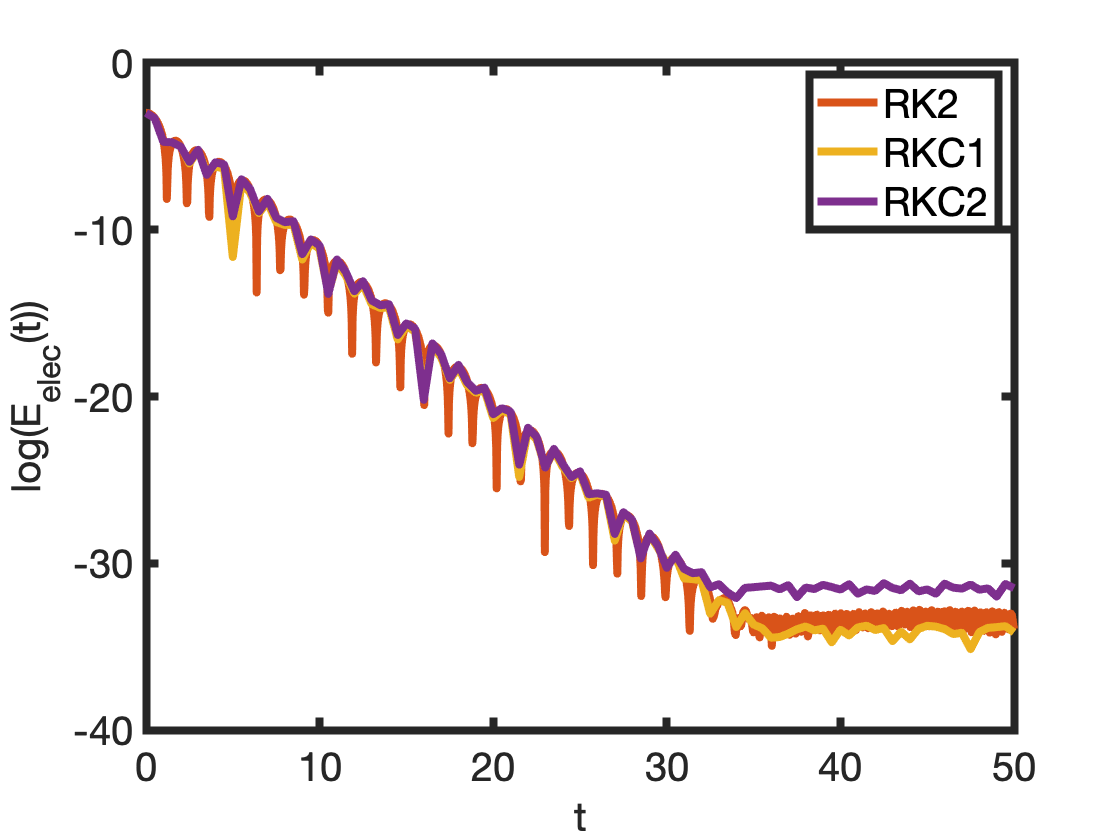}
		\caption{$N_{v_x}=N_{v_y}=128$.}
	\end{subfigure}
	\caption{Test 2 (1dx-2dv): $N_x=32,~v_{\max}=18$. Comparison of the time evolution of the electric energy in semi-log scale for different time integrators: RK2 ($\Delta t=0.05$ (left) and $\Delta t=0.025$ (right)), implicit ($\Delta t=0.5$), RKC2 ($\Delta t=0.5$), RKC1  ($\Delta t=0.5$).}
	\label{fig:5b_enelec}
\end{figure} 

\section{Conclusion}
In this work, we used stabilized Runge-Kutta methods in the context of the numerical simulation of Vlasov-Fokker-Planck equations. 
These time integrators are particularly well adapted to the weakly collisional kinetic equations since they are explicit and allow the 
use of large time steps in contrast to classical implicit methods (which require to solve large linear systems) or explicit methods (which 
need to use small time steps for stability reasons). Regarding the Fokker-Planck operator, a new discretization is proposed based on  
\cite{dellacherie} and a new second time integrator is presented for the time integration of the Vlasov-Poisson-Fokker-Planck equation which 
enables to preserve the total energy, even when coupled with adaptive time step selection strategy. Several tests illustrate the very good behavior of the approach 
(both in terms of accuracy and in terms of computational cost) compared to standard approaches. 

Several extensions like the coupling with the Maxwell equations and taking into account the nonlinear Landau operator 
can be considered.  

\paragraph{Acknowledgements}
The work of I.A. is supported by the Swiss National Science Foundation, grants No. P2GEP2 195212. and P500PT\_210968. The author N.C. has been supported by the EUROfusion Consortium and has received funding from the Euratom research and training programme under grant agreement No. 633053. The views and opinions expressed herein do not necessarily reflect those of the European Commission.

\appendix 

\section{Proof of Proposition \ref{prop_conservation}}
\label{proof_prop}
This part is dedicated to the proof of Proposition \ref{prop_conservation}. 

First, we prove the conservation property. 
Thanks to the fact that  \eqref{form1} can be written in a flux form \eqref{flux_form1}, the mass conservation follows easily. 
For the conservation of momentum, using the definition \eqref{def_UT} of $\tilde{u}$, we have
\begin{eqnarray*}
	\Delta v\sum_j Q_j v_j  &=&  \sum_j v_j \Big[ {f}_{j+1/2} (v_{j+1/2}-\tilde{u}) - {f}_{j-1/2} (v_{j-1/2}-\tilde{u}) + \frac{\tilde{T}}{\Delta v} (f_{j+1} - 2 f_j + f_{j-1})\Big] \nonumber\\
	&=& \sum_j (v_{j+1/2}-\frac{\Delta v}{2}) {f}_{j+1/2} (v_{j+1/2}-\tilde{u}) - \sum_j (v_{j-1/2}+\frac{\Delta v}{2}) {f}_{j-1/2} (v_{j-1/2}-\tilde{u}) \nonumber\\
	&=&  \sum_j v_{j+1/2} {f}_{j+1/2} (v_{j+1/2}-\tilde{u}) -  \sum_j v_{j-1/2} {f}_{j-1/2} (v_{j-1/2}-\tilde{u}) = 0, 
\end{eqnarray*}
where we used the definition \eqref{def_UT} of $\tilde{u}$ and discrete integration by parts.

For the conservation of energy, we insert $v^2_{j} =(v_{j\pm1/2} \mp \Delta v/2)^2$ to get 
\begin{eqnarray*}
	\Delta v \sum_j Q_j \frac{v^2_j}{2}  &=& \frac{1}{2}\sum_j (v^2_{j+1/2}+\frac{\Delta v^2}{4} - v_{j+1/2}\Delta v)  f_{j+1/2} (v_{j+1/2} - \tilde{u})  \nonumber\\
	&& -\frac{1}{2}\sum_j (v^2_{j-1/2}+\frac{\Delta v^2}{4} + v_{j-1/2}\Delta v)  f_{j-1/2} (v_{j-1/2} - \tilde{u})  +   \tilde{T}\Big(\sum_j f_j \Delta v\Big).  
\end{eqnarray*}
Using the definition of $\tilde{u}$, we obtain  
\begin{eqnarray*} 
	\Delta v \sum_j Q_j \frac{v^2_j}{2}  
	&=& - \frac{\Delta v}{2} \sum_j  (v_{j+1/2}   - \tilde{u} +  \tilde{u}) f_{j+1/2} (v_{j+1/2} - \tilde{u})  \nonumber\\ 
	&&  -\frac{\Delta v}{2} \sum_j  (v_{j-1/2}  - \tilde{u} +  \tilde{u}) f_{j-1/2} (v_{j-1/2} - \tilde{u}) + \tilde{T} n\nonumber\\
	&=& -\frac{\Delta v}{2} \sum_j  f_{j+1/2} (v_{j+1/2} - \tilde{u})^2  -\frac{\Delta v}{2} \sum_j   f_{j-1/2} (v_{j-1/2} - \tilde{u})^2 + \tilde{T} n \nonumber\\
	&=& -\Delta v \sum_j  f_{j+1/2} (v_{j+1/2} - \tilde{u})^2 + n\tilde{T},  
\end{eqnarray*}
which is equal to zero thanks to the preservation of momentum and using the definition \eqref{average_f} of $\tilde{T}$.  

Regarding the entropy inequality, thanks to the equivalence between \eqref{form1} and \eqref{form2}, we consider \eqref{form2} to get 
\begin{eqnarray*}
	\sum_{j\in \mathbb{Z}} Q_j  \Big(\frac{f}{M}\Big)_{j} \Delta v &=& \sum_{j\in \mathbb{Z}} \frac{\tilde{T}}{\Delta v^2} \left\{ \tilde{M}_{j+1/2}\Big[\Big(\frac{f}{M}\Big)_{j+1} -\Big(\frac{f}{M}\Big)_{j} \Big] \right\}  \Big(\frac{f}{M}\Big)_{j} \Delta v \nonumber\\
	&&-   \sum_{j\in \mathbb{Z}} \frac{\tilde{T}}{\Delta v^2} \left\{\tilde{M}_{j-1/2}\Big[\Big(\frac{f}{M}\Big)_{j} -\Big(\frac{f}{M}\Big)_{j-1} \Big] \right\}   \Big(\frac{f}{M}\Big)_{j} \Delta v\nonumber\\
	&=&  \sum_{j\in \mathbb{Z}} \frac{\tilde{T}}{\Delta v^2} \left\{ \tilde{M}_{j+1/2}\Big[\Big(\frac{f}{M}\Big)_{j+1} -\Big(\frac{f}{M}\Big)_{j} \Big] \right\}  \Big(\frac{f}{M}\Big)_{j} \Delta v \nonumber\\
	&&-   \sum_{j\in \mathbb{Z}} \frac{\tilde{T}}{\Delta v^2} \left\{\tilde{M}_{j+1/2}\Big[\Big(\frac{f}{M}\Big)_{j+1} -\Big(\frac{f}{M}\Big)_{j} \Big] \right\}   \Big(\frac{f}{M}\Big)_{j+1} \Delta v\nonumber\\
	&=& -\sum_{j\in \mathbb{Z}} \frac{\tilde{T}}{\Delta v^2} \left\{ \tilde{M}_{j+1/2}\Big[\Big(\frac{f}{M}\Big)_{j+1} -\Big(\frac{f}{M}\Big)_{j} \Big]^2 \right\}  \Delta v \leq  0. 
\end{eqnarray*}
To guarantee the last inequality, we need to ensure that $\tilde{M}_{j+1/2}>0$. To do so, we prove in the sequel $\tilde{M}_{j+1/2} = M(v_{j+1/2}) + {\cal O}(\Delta v^2)$ 
(with $M$ given by \eqref{def_maxw}) so that for $\Delta v$ small enough, we can conclude $\tilde{M}_{j+1/2} >0$ since $M(v_{j+1/2})>0$. 



To get the consistency of $\tilde{M}_{j+1/2}$, we consider \eqref{def_M} with $f_j$ replaced by $f(v_j)$ 
and perform Taylor expansions. By an abuse of notations, we will still denote $f(v_j)$ by $f_j$ (same for $M$) 
which allows us to perform Taylor expansions using $f'_j$ for the velocity derivative of $f$ at $v_j$.  

Then, we perform Taylor expansions to reformulate $f_{j+1}M_j - f_jM_{j+1}$ as 
\begin{eqnarray*} 
	f_{j+1}M_j \!-\! f_jM_{j+1}  
	&=&\Delta v (f'_j M_{j+1/2} - f_j M'_{j+1/2} + \frac{\Delta v}{2} (f''_j M_{j+1/2} - f'_j M'_{j+1/2})) +\!{\cal O}(\Delta v^3). 
\end{eqnarray*} 
Using $M_jM_{j+1}+M^2_{j+1/2} + {\cal O}(\Delta v^2)$, the second term in \eqref{def_M} thus becomes 
\begin{eqnarray*} 
	\frac{M_jM_{j+1} (f_{j+1} - f_j)}{f_{j+1}M_j - f_jM_{j+1}} 
	&=&\frac{  M_{j+1/2} (f'_j M_{j+1/2}+ \frac{\Delta v}{2} f''_j M_{j+1/2})+ {\cal O}(\Delta v^3)}{ (f'_j M_{j+1/2} - f_j M'_{j+1/2} + \frac{\Delta v}{2} (f''_j M_{j+1/2} - f'_j M'_{j+1/2})) +\!{\cal O}(\Delta v^3)}. 
\end{eqnarray*} 
For the first term of \eqref{def_M}, using $M'_{j+1/2} = -\Big((v_{j+1/2}-\tilde{u})/\tilde{T}\Big) M_{j+1/2}$, we get  
\begin{eqnarray*} 
	\frac{\Delta v}{\tilde{T}}  \frac{M_jM_{j+1} f_{j+1/2} (v_{j+1/2} - \tilde{u})}{f_{j+1}M_j - f_jM_{j+1}} &=&\frac{(M^2_{j+1/2} + {\cal O}(\Delta v^2))( f_j +  \frac{\Delta v}{2} f'_j + {\cal O}(\Delta v^2)) (v_{j+1/2}-\tilde{u})}{\tilde{T} (f'_j M_{j+1/2} - f_j M'_{j+1/2} + \frac{\Delta v}{2} (f''_j M_{j+1/2} - f'_j M'_{j+1/2})) +\!{\cal O}(\Delta v^3)}\nonumber\\
	&=& \frac{M_{j+1/2}(- f_j M'_{j+1/2}- \frac{\Delta v}{2} f'_j M'_{j+1/2})+ {\cal O}(\Delta v^2) }{ (f'_j M_{j+1/2} - f_j M'_{j+1/2} + \frac{\Delta v}{2} (f''_j M_{j+1/2} - f'_j M'_{j+1/2})) +\!{\cal O}(\Delta v^3)}.  
\end{eqnarray*} 
Gathering now the two terms of \eqref{def_M} leads to 
\begin{eqnarray*} 
	\tilde{M}_{j+1/2} &=&\frac{  M_{j+1/2} (f'_j M_{j+1/2}+ \frac{\Delta v}{2} f''_j M_{j+1/2}- f_j M'_{j+1/2}-  \frac{\Delta v}{2} f'_j M'_{j+1/2})+{\cal O}(\Delta v^2)}{ (f'_j M_{j+1/2} - f_j M'_{j+1/2} + \frac{\Delta v}{2} (f''_j M_{j+1/2} - f'_j M'_{j+1/2})) +\!{\cal O}(\Delta v^3)} \nonumber\\
	&=& M_{j+1/2}  + {\cal O}(\Delta v^2). 
\end{eqnarray*} 

\section{Extension of the Fokker-Planck discretization in the two-dimensional case}
\label{extension_2d}
In this Appendix, we extend the numerical scheme presented in the one-dimensional case to the two-dimensional case. 
We denote $v=(v_x, v_y)$ the velocity variable  so that the Fokker-Planck operator is 
$$
Q(f)(v) = \partial_{v_x} [(v_x-u_{f, x}) f + T\partial_{v_x} f]  + \partial_{v_y} [(v_y-u_{f, y}) f + T_f\partial_{v_y} f],  
$$
where $n_f=\int f dv_x dv_y, n_f u_{f, x/y} = \int v_{x/y} f dv_x dv_y$ and $n_f T_f=\int |v-u|^2 f dv_x dv_y (u=(u_x, u_y))$.  
The velocity space is discretized by $v_{x, i} = i\Delta v_x, v_{y, j} = j\Delta v_y, i, j\in \mathbb{Z}$ and the discretization of the Fokker-Planck operator is 
\begin{eqnarray}
Q(f)(v_{x,i}, v_{y,j}) &\approx& Q_{i,j} \nonumber\\
\hspace{-2cm}&\hspace{-2cm}\hspace{-2cm}&\hspace{-2cm}= \frac{1}{\Delta v_x} \Big[f_{i+1/2,j} (v_{x,i+1/2}-\tilde{u}_x) - f_{i-1/2,j} (v_{x,i-1/2}-\tilde{u}_x)  \Big] + \frac{\tilde{T}}{\Delta v_x^2}(f_{i+1,j}-2f_{i,j}+f_{i-1,j}), \nonumber\\
&& \hspace{-2cm}+ \frac{1}{\Delta v_y} \Big[f_{i,j+1/2} (v_{y,j+1/2}-\tilde{u}_y) - f_{i,j-1/2} (v_{y,j-1/2}-\tilde{u}_y)  \Big] + \frac{\tilde{T}}{\Delta v_y^2}(f_{i,j+1}-2f_{i,j}+f_{i,j+1})\nonumber
\end{eqnarray}
with $f_{i+1/2,j} = (f_{i+1,j} + f_{i,j})/2$, $f_{i, j+1/2} = (f_{i,j+1} + f_{i,j})/2$ and 
\begin{eqnarray} 
n&=& \sum_{i,j} f_{i,j} \Delta v_x \Delta v_y, \nonumber\\ 
n\tilde{u}_x &=& \sum_{i,j} v_{x,i+1/2}f_{i+1/2,j} \Delta v_x \Delta v_y, \;\;\;\; n\tilde{u}_y = \sum_{i,j} v_{y,j+1/2}f_{i,j+1/2} \Delta v_x \Delta v_y, \nonumber\\ 
2n\tilde{T}&=& \sum_{i,j} (v_{x,i+1/2}-\tilde{u}_x)^2 f_{i+1/2,j} \Delta v_x \Delta v_y +\sum_{i,j} (v_{y,j+1/2} - \tilde{u}_y)^2f_{i,j+1/2} \Delta v_x \Delta v_y, 
\end{eqnarray}
where $v_{x,i+1/2} = (v_{x, i}+v_{x, i+1})/2$, $v_{y,j+1/2} = (v_{y, j}+v_{y, j+1})/2$. 
Similar expressions are obtained for the fourth order discretization.

\bibliographystyle{abbrv}
\bibliography{refs}

\end{document}